\documentclass[twoside]{amsart}
\usepackage{amsmath}
\usepackage{amssymb}
\usepackage{amsfonts}
\usepackage{amsthm}
\usepackage{color}

\theoremstyle{plain}

\newtheorem{theorem}{Theorem}[section]
\newtheorem{mtheorem}[theorem]{Main Theorem}

\newtheorem{corollary}[theorem]{Corollary}

\newtheorem{definition}[theorem]{Definition}
\newtheorem{example}[theorem]{Example}

\newtheorem{lemma}[theorem]{Lemma}
\newtheorem{notation}[theorem]{Notation}

\newtheorem{proposition}[theorem]{Proposition}
\newtheorem{remark}[theorem]{Remark}

\numberwithin{equation}{section}

\begin{document}
\title[Local Automorphisms]{Local automorphisms of finitary incidence algebras}

\author{Jordan Courtemanche}
\address[Jordan Courtemanche]{Department of Mathematics, Baylor University, One Bear Place \#97328, Waco, TX 76798-7328, USA}
\email{Jordan\_Courtemanche@baylor.edu.}

\author{Manfred Dugas}
\address[Manfred Dugas]{Department of Mathematics, Baylor University, One Bear Place \#97328, Waco, TX 76798-7328, USA}
\email{Manfred\_Dugas@baylor.edu}

\author{Daniel Herden}
\address[Daniel Herden]{Department of Mathematics, Baylor University, One Bear Place \#97328, Waco, TX 76798-7328, USA}
\email{Daniel\_Herden@baylor.edu}

\thanks{Corresponding author: Manfred Dugas}

\begin{abstract}
Let $R$ be a commutative, indecomposable ring with identity and $(P,\le)$ a partially ordered set.
Let $FI(P)$ denote the finitary incidence algebra of $(P,\le)$ over $R$. We will
show that, in most cases, local automorphisms of $FI(P)$ are actually $R$-algebra automorphisms.
In fact, the existence of local automorphisms which fail to be $R$-algebra automorphisms will depend
on the chosen model of set theory and will require the existence of measurable cardinals. We will
discuss local automorphisms of cartesian products as a special case in preparation of the general
result.
\end{abstract}

\maketitle

\tableofcontents

\section{Introduction}

Let $R$ be a commutative ring with identity and $A$ some $R$-algebra. Let $n$
denote a positive integer. The $R$-linear map $\eta:A\rightarrow A$ is called an
\emph{$n$-local automorphism} if for any elements $a_{i}\in A$, $1\leq i\leq n$,
there exists some $R$-algebra automorphism $\varphi: A\rightarrow A$
such that $\eta(a_{i})=\varphi(a_{i})$ for all $1\leq i\leq n$. Any $1$-local
automorphism is simply called a \emph{local automorphism}. Local automorphisms of algebras
have attracted attention over the years. For example, in 1990, Larson and
Sourour \cite{LS} proved that if $X$ is an
infinite-dimensional Banach space and $\eta$ is a local automorphism of
$\mathfrak{B}(X)$, the algebra of all bounded linear operators on $X$, then
$\eta$ is an automorphism of $\mathfrak{B}(X)$. In 1999, Crist \cite{C} studied the
local automorphisms of finite dimensional CSL algebras. In 2003, Hadwin and Li \cite{HL}
showed that any surjective $2$-local automorphism of a nest algebra $\mathcal{A}$
is an automorphism. We refer to the introductions and references of these three
papers for more details on the history and results regarding local automorphisms of algebras.
For further reading, see also \cite{BS1,BS2,S}.

From now on, we will always assume that the commutative ring $R$ is
indecomposable, i.e., $0$ and $1$ are the only idempotent elements of $R$. Let
$(P,\le)$ be a poset. Generalizing the notion of the incidence algebra \cite{R,SComb} of a
locally finite poset $(P,\le)$ over $R$, Khripchenko and Novikov \cite{KN} introduced the concept
of the \emph{finitary incidence algebra} $FI(P)$ for arbitrary posets $(P,\le)$ in 2009. One year
later, Khripchenko showed that the algebra $FI(P)$ has exactly the same kind of
automorphisms as in the case of a locally finite poset $(P,\le)$, cf. \cite{Khrip,Stanley}.
Our goal is to determine the local automorphisms of $FI(P)$. The proof will introduce a number
of new ideas and concepts for finitary incidence algebras, combining methods from set theory,
combinatorics and algebra.

We will associate a cardinal $\mu_{R}$ to the ring $R$ as follows: If $R$ is
finite, then $\mu_{R}=\aleph_{0}$. If $R$ is infinite, then $\mu_{R}$ is the
least cardinal that allows a non-trivial $\left\vert R\right\vert ^{+}$-additive
measure with values $0$ and $1$. Let $\operatorname{Aut}(A)$ denote the group of all $R$-algebra
automorphisms of the $R$-algebra $A$, which is a subset of $\operatorname{LAut}(A)$,
the set of all local automorphisms of $A$. Our main result is the following:

\begin{mtheorem} \label{thm:main0}
Let $(P,\le)$ be a poset and $R$ an indecomposable ring.
\begin{itemize}
\item[(a)] If $\left\vert P\right\vert <\mu_{R}$, then $\operatorname{LAut}(FI(P))=\operatorname{Aut}(FI(P))$.
\item[(b)] If $\left\vert P\right\vert \geq\mu_{R}$, then non-surjective local automorphisms may exist,
and examples of local automorphisms which are not $R$-algebra automorphisms are provided.
\item[(c)] If $\eta\in \operatorname{LAut}(FI(P))$ is surjective and $\left\vert R\right\vert \geq 3$, then $\eta\in \operatorname{Aut}(FI(P))$.
\end{itemize}
\end{mtheorem}

Let $\mu_0$ denote the least measurable cardinal, if it exists. Measurable cardinals are large cardinals
known to not exist in many set-theoretic models of ZFC, for example $V=L$,
and it cannot be proven that the existence of measurable
cardinals is consistent with ZFC. We refer to \cite{Jech} for more details on
measurable cardinal. For $\aleph_{0}\leq\left\vert R\right\vert <\mu_0$, we have $\mu_{R}=\mu_0$,
and we note the following immediate consequence of Theorem \ref{thm:main0}(a).

\begin{corollary}
Let $R$ be an infinite indecomposable ring and $(P,\le)$ a poset such that
$\left\vert P\right\vert <\mu_0$. Then $\operatorname{LAut}(FI(P))=\operatorname{Aut}(FI(P))$.
\end{corollary}

If $(P,\le)$ is a trivial poset, i.e., the partial order is just equality, then
$\Pi=FI(P)=\prod_{x\in P} Re_{x}$ is just the full cartesian product
of $\left\vert P\right\vert $ copies of $R$. Note that for any poset $(P,\le)$,
the $R$-algebra $\Pi$ is naturally an epimorphic image of $FI(P)$. This motivates
the study of $\operatorname{Aut}(\Pi)$ and $\operatorname{LAut}(\Pi)$ as preparation
for a proof of Theorem \ref{thm:main0}.

For any permutation $\rho: P \to P$, the map $\widehat{\rho}:\Pi \rightarrow\Pi$ with
$\widehat{\rho}\left({\sum}_{x\in P}r_{x}e_{x}\right)={\sum}_{x\in P}r_{x}e_{\rho(x)}$ is in
$\operatorname{Aut}(\Pi)$. Our starting point will be the simple observation that every $R$-algebra
automorphisms of $\Pi$ arises like this from a suitable permutation $\rho$ of $P$.

\begin{itemize}
\item[] {\bf Proposition \ref{prop:inducedauto}.}
\emph{$\operatorname{Aut(\Pi)} =\{ \widehat\rho: \rho \mbox{ is a permutation of $P$}\}.$}
\end{itemize}

From here we set out to discuss local automorphisms of $\Pi$ and their properties.
By nature, local automorphisms will be very close to $R$-algebra automorphisms, and the question
whether, or when, $\operatorname{Aut}(\Pi)=\operatorname{LAut}(\Pi)$ holds will be pursued systematically
in Section \ref{sec:cartprod}. For $2$-local automorphisms we will show the following.

\begin{itemize}
\item[] {\bf Corollary \ref{cor:sur2localauto}.}
\emph{If $\eta \in \operatorname{LAut}(\Pi)$ is a surjective 2-local automorphism, then $\eta \in \operatorname{Aut}(\Pi)$.}
\end{itemize}

In addition, we have a correlating result for local automorphisms.

\begin{itemize}
\item[] {\bf Theorem \ref{thm:surjectivelautisaut}.}
\emph{Let $\left\vert R\right\vert \geq3$ and $\eta\in \operatorname{LAut}(\Pi)$ such that
$e_{x}\in\operatorname{Im}(\eta)$, the image of $\eta$, for all $x\in P$. Then $\eta\in \operatorname{Aut}(\Pi)$.}
\end{itemize}

Thus, under some weak assumption of surjectivity, local automorphisms indeed coincide with $R$-algebra automorphisms.
The case $\left\vert R\right\vert =2$, i.e., $R=\mathbb{Z}_2$, however, is special.

\begin{itemize}
\item[] {\bf Theorem \ref{thm:R=2nonsurjlaut}.}
\emph{Let $R=\mathbb{Z}_2$ and $\kappa=\left\vert P\right\vert $ an infinite
cardinal. Then there exist $2^{2^{\kappa}}$-many non-surjective local
automorphisms of $\Pi$, as well as $2^{2^{\kappa}}$-many bijective local
automorphisms of $\Pi$ that {\bf are not} $R$-algebra automorphisms of $\Pi$.}
\end{itemize}

Our next main result shows that ``$n$-local'' does, in general, not imply ``surjective'' for local automorphisms:

\begin{itemize}
\item[] {\bf Lemma \ref{lem:finiteRnonsurjective}.}
\emph{Let $R$ be finite and $P$ be countably infinite. Then there exists some
$\eta\in \operatorname{LAut}(\Pi)$ such that $\eta$ is an $n$-local automorphism for all natural numbers $n$,
but $\eta$ is {\bf not} surjective. In particular, $\eta\notin \operatorname{Aut}(\Pi)$.}
\end{itemize}

This result holds independently of the chosen model of ZFC and generalizes to infinite rings $R$ as follows.

\begin{mtheorem} \label{thm:main1}\mbox{}
\begin{itemize}
\item[(a)] $\operatorname{LAut}(\Pi)=\operatorname{Aut}(\Pi)$ if and only if $\left\vert P\right\vert <\mu_{R}$. Moreover,
\item[(b)] Let $\left\vert P\right\vert \geq\mu_{R}$. Then there exists some unital
$R$-linear map $\eta$ such that $\eta$ is an $n$-local automorphism for all
natural numbers $n$, but $\eta$ is {\bf not} surjective. In particular, $\eta\notin \operatorname{Aut}(\Pi)$.
\end{itemize}
\end{mtheorem}

This theorem is notable, as it links the existence of nontrivial local automorphisms $\eta\notin \operatorname{Aut}(\Pi)$ to the existence
of measurable cardinals. Therefore, the existence of nontrivial local automorphisms depends on the chosen model of set theory.

Theorem \ref{thm:main1} extends to $\operatorname{End}(\Pi)$, the set of $R$-algebra endomorphisms of $\Pi$.
For a subset $I$ of $P$, we define $e_{I}={\sum}_{x\in I}e_{x}$, and we call $\eta\in \operatorname{End}(\Pi)$
\emph{induced} if there is a family $\{A_{x}:x\in P\}$ of pairwise disjoint subsets of $P$, such that
$\eta({\sum}_{x\in P}a_{x}e_{x})  ={\sum}_{x\in P}a_{x}e_{A_{x}}$ for all ${\sum}_{x\in P}a_{x}e_{x}\in\Pi$. We have
the following result.

\begin{itemize}
\item[] {\bf Theorem \ref{thm:mainendo}.}
\emph{Let $R$ be a field. Then every $\eta\in \operatorname{End}(\Pi)$ is induced if and
only if $\left\vert P\right\vert <\mu_{R}$.}
\end{itemize}

Theorem \ref{thm:main1} will serve as a stepping stone for proving Theorem \ref{thm:main0} in Section~\ref{sec:fips}.
With Section \ref{sec:fipsdef} we include a short overview of finitary incidence algebras. Our notation is standard,
cf. \cite{SOD}.

\section{Finitary incidence algebras} \label{sec:fipsdef}

For the convenience of the reader, we recall some definitions and results.
Let $(P,\leq)$ denote a poset and $R$ a commutative ring with $1$.

\begin{definition}[Khripchenko, Novikov, \protect\cite{KN}]
\mbox{}

\begin{itemize}
\item[(a)] We call $I(P)=\prod_{x,y\in P,\,x\leq y}Re_{xy}$ the induced
\emph{incidence space }of $(P,\le)$ over $R$.
$I(P)$ is an $R$-module under componentwise addition and scalar
multiplication.

\item[(b)] Let
\begin{align*}
& FI(P)=\{\,a=\sum\limits_{x\leq y}a_{xy}e_{xy}\in I(P): \\
& \{(u,v):x\leq u<v\leq y\mbox{ and }a_{uv}\not=0\}
\mbox{ is finite for all }x<y\}.
\end{align*}
Then $FI(P)$ is an $R$-algebra, the induced \emph{finitary incidence algebra of}
\mbox{$(P,\le)$} over $R$. The multiplication (of formal sums) is induced by
\[
e_{xu}e_{vy}=\left\{
\begin{array}{rl}
e_{xy}, & \mbox{ if }x\leq u=v\leq y, \\
0, & \mbox{ otherwise.}
\end{array}
\right.
\]
\end{itemize}
\end{definition}

Note that $I(P)$ is a bimodule over the ring $FI(P)$ and
$a=\sum_{x\leq y}a_{xy}e_{xy}\in FI(P)$ is a unit in $FI(P)$ if and only if all
coefficients $a_{xx}$ are units in $R$. Moreover,
\[
Z:=Z(FI(P))=\{
a=\textstyle \sum_{x\leq y}a_{xy}e_{xy}\in FI(P):a_{xx}=0\text{ \textit{for all }} x\in P\}
\]
is a two-sided ideal of $FI(P)$ and $FI(P)/Z\cong \prod_{x\in P}Re_{xx}$.

\begin{definition}
\mbox{}
\begin{itemize}
\item[(a)] For any invertible $f\in FI(P)$ let $\psi _{f}$ be the induced
\emph{inner automorphism}, and $\operatorname{Inn}(FI(P))$ denotes the set of all inner
automorphisms of $FI(P)$. Note that $\operatorname{Inn}(FI(P))$ is a normal subgroup of
$\operatorname{Aut}(FI(P))$.

\item[(b)] For any order automorphism $\rho $ of $(P,\leq )$ let $\widehat{\rho }$
denote the induced $R$-algebra automorphism of $FI(P)$ with
\[
\widehat{\rho }\left( \sum_{x\leq y}a_{xy}e_{xy}\right) =\sum_{x\leq
y}a_{xy}e_{\rho (x)\rho (y)}.
\]
Let $\operatorname{Aut}(P)=\{\widehat{\rho}:\rho \mbox{ order automorphism of } P\}$,
a subgroup of $\operatorname{Aut}(FI(P))$.
\item[(c)] Given units $\sigma _{xy}\in R$ with $\sigma _{xy}\sigma
_{yz}=\sigma _{xz}$ for all $x\leq y\leq z$ let $M_{\sigma }$ denote the
induced $R$-algebra automorphism \emph{(Schur multiplication)} of $FI(P)$,
\[
M_{\sigma }\left( \sum_{x\leq y}a_{xy}e_{xy}\right) =\sum_{x\leq
y}a_{xy}\sigma _{xy}e_{xy}.
\]
Let $\operatorname{Mult}(FI(P))$ denote the set of all Schur multiplications of $%
FI(P)$. Then $\operatorname{Mult}(FI(P))$ is a subgroup of $\operatorname{Aut}(FI(P))$.
\end{itemize}
\end{definition}

Quite frequently, we will apply the following result:

\begin{theorem}[Khripchenko, \protect\cite{Khrip}]
\label{thm:charaut} Let $(P,\leq )$ be a poset and $R$ a commutative ring
with $1$. If $R$ is indecomposable, then every $\varphi \in \operatorname{Aut}(FI(P))$
decomposes canonically:
\[
\varphi =\psi _{f}\circ M_{\sigma }\circ \widehat{\rho }.
\]
\end{theorem}

For the original result on incidence algebras, cf. \cite{SOD,Stanley}.\medskip

The theorem implies that $\operatorname{Aut}(FI(P))$ is a product of the subgroups
$\operatorname{Inn}(FI(P))$, $\operatorname{Mult}(FI(P))$ and $\operatorname{Aut}(P)$.
It follows from the definition, that $\operatorname{Aut}(P)$ normalizes $\operatorname{Mult}(FI(P))$,
which implies that $\operatorname{Aut}(FI(P))$ is the product of these three subgroups in
any order of the factors. Note that, in general,
\[
\operatorname{Inn}(FI(P))\cap \operatorname{Mult}(FI(P))\neq \{\operatorname{id}_{FI(P)}\},
\]
but $\operatorname{Aut}(FI(P))$ is a semidirect product of $\operatorname{Inn}(FI(P)) \operatorname{Mult}(FI(P))$ by $\operatorname{Aut}(P)$.
Moreover, the subgroup
\[
\operatorname{Inn}(FI(P)) \operatorname{Mult}(FI(P))\subseteq \operatorname{Aut}(FI(P))
\]
corresponds to those $R$-algebra automorphisms of $FI(P)$ which induce the identity map on $FI(P)/Z$.
One of the many consequences is that the ideal $Z$ is invariant under all $R$-algebra automorphisms of $FI(P)$.

Let $D(P)=FI(P)\oplus I(P)$ denote the idealization (Dorroh extension) of $I(P)$ by $FI(P)$.
A complete description of $\operatorname{Aut}(D(P))$ was obtained in \cite{DW}.

\section{Local automorphisms of cartesian products} \label{sec:cartprod}

In Section \ref{sec:cartprod} we will investigate the structure of local automorphisms
and $R$-algebra endomorphisms of cartesian products. We fix the following notations and definitions.

\begin{notation}
Let $P$ be a set and $R$ a commutative, indecomposable ring, i.e., $0$ and $1$
are the only idempotent elements of the ring $R$. Let $\Pi=\prod_{x\in P}Re_{x}$
be the cartesian product. Then $\Pi$ is an $R$-algebra, and by
$\operatorname{End}(\Pi)$ and $\operatorname{Aut}(\Pi)$ we denote the set of all $R$-algebra
endomorphisms and automorphisms, respectively, of $\Pi$.
We will associate any element $a \in \Pi$ with its standard representation $a=\sum_{x\in P}a_{x}e_{x}$.
For $X\subseteq P$ define $e_{X}=\sum_{x\in X}e_{x}$.
Let $\rho :P\rightarrow P$ be any map. Define $\widehat{\rho}:\Pi\rightarrow\Pi$ by
$\widehat{\rho}\:(\sum_{x\in P}a_{x}e_{x})=\sum_{x\in P}a_{x}e_{\rho(x)}$ for all $a\in \Pi$.
It is readily verified that $\widehat{\rho}$ is an $R$-algebra endomorphism of $\Pi$ if and
only if $\rho$ is injective.
Note that $\widehat{\rho}$ may not be unital.
\end{notation}

We record a first easy observation on $\widehat \rho$.

\begin{proposition} \label{prop:inducedauto}
Let $R$ be an indecomposable ring. Then for $\Pi=\prod_{x\in P}Re_{x}$ holds
\[
\operatorname{Aut(\Pi)} =\{ \widehat\rho: \rho \mbox{ is a permutation of $P$}\}.
\]
\end{proposition}

\begin{proof}
If $\rho$ is a permutation of $P$, then $\widehat{\rho^{-1}}=\widehat{\rho}^{-1}$, and
thus $\widehat{\rho}$ is an automorphism.
Now assume that $\eta\in \operatorname{Aut(\Pi)}$ is an automorphism and let $E=\{e_{x}:x\in P\}$. Then $E$
is the set of all primitive idempotents of $\Pi$ and thus $\eta(E)\subseteq E$
as well as $\eta^{-1}(E)\subseteq E$. This shows that
${\eta\upharpoonright}_E: E\rightarrow E$ is bijective and we infer that there exists a permutation
$\rho$ of $P$ such that $\eta(e_{x})=e_{\rho(x)}$ for all $x\in P$. Let
$a=\sum_{x\in P}a_{x}e_{x}\in\Pi$. Then
\[
\eta (a) e_{\rho(x)} =\eta(a)\eta(e_{x}) =\eta(a e_{x}) =\eta(a_x e_{x}) =a_x\eta(e_{x}) =a_{x}e_{\rho(x)}
\]
for all $x\in P$. Thus $a_{x}$ is the entry of the $e_{\rho(x)}$-coordinate of $\eta(a)$.
This shows that $\eta=\widehat{\rho}$.
\end{proof}

We continue with the central definition of this section.

\begin{definition} \mbox{}
\begin{itemize}
\item[(a)] An $R$-linear map $\eta: FI(P) \rightarrow FI(P)$ is called an \emph{$n$-local automorphism}
if for every choice of $a_i \in FI(P)$, $1\le i\le n$, there exists some $R$-algebra automorphism
$\varphi \in \operatorname{Aut}(FI(P))$ with $\eta(a_i)=\varphi(a_i)$ for all $1\le i\le n$.
\item[(b)] We will refer to $1$-local automorphisms simply as \emph{local automorphisms}.
With $\operatorname{LAut}(FI(P))$ we denote the set of local automorphisms.
\end{itemize}
\end{definition}

\begin{remark}
Every $R$-algebra automorphism is an $n$-local automorphism for all $n>0$.
However, a local automorphism does not need to be an $R$-algebra homomorphism.
Clearly, $n$-local automorphisms will preserve multiplication for $n\ge 3$.
\end{remark}

We note some general properties of local automorphisms.

\begin{proposition} \label{prop:triviallautoproperties}
Let $\eta$ be a local automorphism of an $R$-algebra $A$. Then the following holds:
\begin{itemize}
\item[(a)] $\eta$ is injective.
\item[(b)] $\eta$ preserves idempotents.
\item[(c)] $\eta$ preserves primitive idempotents.
\end{itemize}
\end{proposition}

\begin{proof}
Let $a \in \operatorname{Ker}(\eta)$. Then $a \in \operatorname{Ker}(\varphi)$ for some
$\varphi \in \operatorname{Aut(A)}$. Thus $\operatorname{Ker}(\eta)=0$, and (a) holds.
Similarly, as $R$-algebra automorphisms preserve idempotents and primitive idempotents,
(b) and (c) hold.
\end{proof}

\subsection{Surjective local automorphisms} \label{sec:surlocalautos}

In this section we present some first results on surjective local automorphisms.
We start with the main result.

\begin{theorem} \label{thm:surjectivelautisaut}
Let $R$ be an indecomposable ring with $|R|\ge 3$, and $\eta$
a local automorphism of $\Pi=\prod_{x\in P}Re_{x}$ such that $e_{x}\in \operatorname{Im}(\eta)$ for all $x\in P$.
Then $\eta$ is an $R$-algebra automorphism, and $\eta=\widehat{\rho}$ for some permutation
$\rho:P\rightarrow P$.
\end{theorem}

\begin{proof}
Let $E=\{e_{x}:x\in P\}$ be the set of all primitive idempotents of $\Pi$.
Then $\eta(E)\subseteq E$, and $E$ is closed under preimages under $\eta$.
We infer that $\eta(e_{x})=e_{\rho(x)}$ defines a permutation $\rho:P\rightarrow P$.

Fix some $y\in P$ and consider an element $a=\sum_{x\in P} a_{x}e_{x}\in \Pi$ such that
\begin{align} \label{eq:ycoordinate1}
a_{x}\neq a_{y} \mbox{\quad for all } y\neq x\in P.
\end{align}
Since $\eta$ is a local automorphism of $\Pi$, there exist permutations
$\sigma, \tau:P\rightarrow P$ such that
\[
\eta(a)=\widehat\sigma(a)=\sum_{x\in P}a_{x}e_{\sigma(x)} \mbox{ and }
\eta(a+e_{y})=\widehat\tau(a+e_{y})=(a_y+1)e_{\tau(y)}+\sum_{y\not=x\in P}a_{x}e_{\tau(x)}.
\]
In particular,
\[
(a_y+1)e_{\tau(y)}+\sum_{y\not=x\in P}a_{x}e_{\tau(x)} = \eta(a+e_{y}) = \eta(a)+ \eta(e_{y})= e_{\rho(y)}+\sum_{x\in P}a_{x}e_{\sigma(x)}.
\]
For $\rho(y)\neq\sigma(y)$, note that $a_y$ is the $e_{\sigma(y)}$-coordinate of the right-hand side, while all coordinates on the
left-hand side differ from $a_y$. Thus, we infer $\rho(y)=\sigma(y)$ or, in terms of coordinate entries of $a$ and $\eta(a)$,
\begin{align} \label{eq:ycoordinate2}
(\eta(a))_{\rho(y)}=(\eta(a))_{\sigma(y)}=a_y.
\end{align}

Now let $a \in\Pi$ be arbitrary. If we can decompose $a=b+c$ in such a way that $b,c \in \Pi$ both satisfy Property \eqref{eq:ycoordinate1},
then \eqref{eq:ycoordinate2} applies to $b$ and $c$, and
\begin{align} \label{eq:ycoordinate3}
(\eta(a))_{\rho(y)}=(\eta(b))_{\rho(y)}+(\eta(c))_{\rho (y)}=b_{y}+c_{y}=a_{y}=(\widehat\rho(a))_{\rho(y)}.
\end{align}
We define $b\in\Pi$ by choosing $b_{y}=0$ and $b_{x}\in R\setminus \{0,a_{x}-a_{y}\}$ for all $y \not= x\in P$.
Let $c=a-b$. Then $c_{y}=a_{y}$ and
\[
c_{x}=a_{x}-b_{x}\notin\{a_{x}-0,a_{x}-(a_{x}-a_{y})\}=\{a_{x},a_{y}\}=\{a_{x},c_{y}\}.
\]
for all $y \not= x\in P$. Thus, Property \eqref{eq:ycoordinate1} holds for both $b$ and $c$.

Note that \eqref{eq:ycoordinate3} holds for all $y\in P$, thus $\eta(a)=\widehat \rho(a)$ for all $a \in \Pi$,
and $\eta = \widehat \rho$.
\end{proof}

We note an important immediate consequence of this theorem.

\begin{corollary} \label{cor:surjectivelautisaut}
Let $R$ be an indecomposable ring with $|R|\ge 3$, and $\eta$
a surjective local automorphism of $\Pi=\prod_{x\in P}Re_{x}$.
Then $\eta$ is an $R$-algebra automorphism.
\end{corollary}

Theorem \ref{thm:surjectivelautisaut} and Corollary \ref{cor:surjectivelautisaut}
fail in the case of $|R|=2$ as will be detailed in Section \ref{sec:R=2}.
In this respect, the $2$-local automorphisms of $\Pi$ are better behaved.

\begin{theorem}
Let $R$ be an indecomposable ring, and $\eta$
a $2$-local automorphism of $\Pi=\prod_{x\in P}Re_{x}$ such that $e_{x}\in \operatorname{Im}(\eta)$ for all $x\in P$.
Then $\eta$ is an $R$-algebra automorphism, and $\eta=\widehat{\rho}$ for some permutation
$\rho:P\rightarrow P$.
\end{theorem}

\begin{proof}
As seen in Theorem \ref{thm:surjectivelautisaut}, $\eta$ induces a permutation
$\rho:P\rightarrow P$ with $\eta(e_{x})=e_{\rho(x)}$ for all $x\in P$. Let
$a=\sum_{x\in P} a_{x}e_{x}\in \Pi$ and $y\in P$. Since $\eta$ is a $2$-local
automorphism, there exists a permutation $\sigma:P\rightarrow P$ such that
$\eta(a)=\widehat{\sigma}(a)$ and $\eta(e_{y})=\widehat{\sigma}(e_{y})$.
It follows that $e_{\rho(y)}=\eta(e_{y})=\widehat{\sigma}(e_{y})=e_{\sigma(y)}$,
and thus $\rho(y)=\sigma(y)$. Now we have
\begin{align} \label{eq:ycoordinate4}
(\eta (a))_{\rho(y)}=(\eta (a))_{\sigma(y)}=(\widehat{\sigma}(a))_{\sigma(y)}=a_{y}=(\widehat\rho(a))_{\rho(y)}.
\end{align}

Note that \eqref{eq:ycoordinate4} holds for all $y\in P$, thus $\eta(a)=\widehat \rho(a)$
for all $a \in \Pi$, and $\eta = \widehat \rho$.
\end{proof}

We mention the following counterpart to Corollary \ref{cor:surjectivelautisaut}.

\begin{corollary} \label{cor:sur2localauto}
Let $R$ be an indecomposable ring, and $\eta$
a surjective $2$-local automorphism of $\Pi=\prod_{x\in P}Re_{x}$.
Then $\eta$ is an $R$-algebra automorphism.
\end{corollary}

\subsection{The exceptional case $|R|=2$} \label{sec:R=2}

Now we consider the special case $|R| =2$, which was left out in Theorem \ref{thm:surjectivelautisaut}.
In this case,
$R=\mathbb{Z}_2=\{0,1\}$ is the field with two elements, and the $R$-algebra $\Pi=(\Pi,+,\cdot)$ is
naturally isomorphic to $\mathfrak{B}=(\mathcal{P}(P),\triangle,\cap)$, where
$\mathcal{P}(P)$ is the power set of $P$, and $\triangle$ denotes the symmetric
difference. Let $\varphi$ be an automorphism of $\mathfrak{B}$. By Proposition \ref{prop:inducedauto}
there exists some permutation $\rho:P\rightarrow P$ such that $\varphi
=\widehat{\rho}$, i.e.,
\[
\varphi(X)=\{\rho(x):x\in X\}=\rho(X) \mbox{\: and \:} P\setminus \varphi(X)=\{\rho(x):x\in P\setminus X\}=\rho(P\setminus X)
\]
for all subsets $X$ of $P$. Note, in particular, that $\varphi$ preserves cardinalities,
\[
|\varphi(X)| =|X| \mbox{\quad and \quad} |P\setminus \varphi(X)| =|P\setminus X| \mbox{\quad for all } X\subseteq P.
\]
This observation motivates the following characterization of local automorphisms.

\begin{proposition}
Let $\eta:\mathfrak{B}\rightarrow\mathfrak{B}$ be a map. Then $\eta$ is a
local automorphism of $\mathfrak{B}$ if and only if
\begin{align} \label{eq:R=2char}
|\eta(X)| =|X| \mbox{\quad and \quad} |P\setminus \eta(X)| =|P\setminus X| \mbox{\quad for all } X\subseteq P.
\end{align}
\end{proposition}

\begin{proof}
Assume that $\eta:\mathfrak{B}\rightarrow\mathfrak{B}$ is a local
automorphism of $\mathfrak{B}$, and $X\subseteq P$. Then there exists an
automorphism $\varphi:\mathfrak{B}\rightarrow\mathfrak{B}$ with
$\eta(X)=\varphi(X)$, which yields $|\eta(X)|=|\varphi(X)|=|X|$ and $|P\setminus \eta(X)|=|P\setminus \varphi(X)|=|P\setminus X|$.

Now assume that $\eta:\mathfrak{B}\rightarrow\mathfrak{B}$ is a map with
$|\eta(X)| =|X|$ and $|P\setminus \eta(X)| =|P\setminus X|$ for all $X\subseteq P$.
Thus, for any given $X\subseteq P$ we can choose some permutation
$\rho:P\rightarrow P$ with $\rho(X)=\eta(X)$.
Then $\widehat{\rho}$ is an automorphism of $\mathfrak{B}$ such that
$\eta(X)=\widehat{\rho}(X)$, and $\eta$ is a local automorphism.
\end{proof}

Property \eqref{eq:R=2char} provides a simple and powerful tool for
constructing various  local automorphisms. We will use it to give
counterexamples for Theorem~\ref{thm:surjectivelautisaut} and
Corollary \ref{cor:surjectivelautisaut}, namely, bijective local automorphisms
of $\mathfrak{B}$ which are no automorphisms.

\begin{theorem}  \label{thm:R=2nonsurjlaut}
Let $\kappa=|P|$ be infinite. Then there exist
$2^{2^{\kappa}}$-many non-surjective local automorphisms of $\mathfrak{B}$, as
well as $2^{2^{\kappa}}$-many bijective local automorphisms of $\mathfrak{B}$
that are not $R$-algebra automorphisms of $\mathfrak{B}$.
\end{theorem}

\begin{proof}
Let
\[
K=\{X\subseteq P: |X| <\kappa \mbox{ or } |P\setminus X| <\kappa\}.
\]
It is easy to see that $K$ is closed with respect
to $\triangle$ and thus a subgroup of cardinality $|K|=2^{<\kappa}$ of the elementary
abelian $2$-group $\mathfrak{B}=(\mathcal{P}(P),\triangle)$. Thus, there is a subgroup $C$ of
$\mathfrak{B}$ such that $\mathfrak{B}=K\oplus C$ with respect to $\triangle$. We claim that
\[
|C|=|\mathfrak{B}|=2^\kappa.
\]
To see this, let us call a family $\mathcal{F}$ of subsets of $P$ \emph{independent}
if for any distinct sets $X_1,\dots,X_n,Y_1,\ldots,Y_m$ in $\mathcal{F}$ the intersection
\[
\bigcap_{i=1}^n X_i \cap \bigcap_{j=1}^m (P \setminus Y_j)
\]
has cardinality $\kappa$. With \cite[Lemma 7.7]{Jech} there exists an independent family
$\mathcal{F}\subseteq \mathcal{P}(P)$ of cardinality $|\mathcal{F}|=2^\kappa$,
and it is easy to check that $\mathcal{F}$ is a set of independent elements modulo $K$.
Hence $|C|=|\mathcal{F}|=2^\kappa$.

Now let $\theta:C\rightarrow C$ be any injective $\triangle$-homomorphism,
and define $\eta:\mathfrak{B}\rightarrow\mathfrak{B}$ by
$\eta=\operatorname{id}_{K}\oplus\, \theta$. Then $\eta$ is a
homomorphism with respect to $\triangle$ and with \eqref{eq:R=2char}
a local automorphism of $\mathfrak{B}$.

If $\theta$ is chosen to be non-surjective, then $\eta$ is a non-surjective local automorphism,
and there are $2^{2^{\kappa}}$ such maps. Of course, if $\theta$ is chosen to be bijective,
then $\eta$ is a bijective local automorphism, and there are again $2^{2^{\kappa}}$ such maps.
Note, however, that $\mathfrak{B}$ has only $2^{\kappa}$
many automorphisms since those are induced by permutations $\rho:P\rightarrow P$, cf. Proposition \ref{prop:inducedauto}.
\end{proof}

In the last theorem we caught a first glimpse of non-surjective local automorphisms.
These will become the topic of Section \ref{sec:nonsurjlaut}. Note also, that independent
families, as introduced in the proof of Theorem \ref{thm:R=2nonsurjlaut}, closely
relate to ultrafilters. This connection will intensify in the following section.

\subsection{Non-surjective local automorphisms} \label{sec:nonsurjlaut}
As seen in Section \ref{sec:surlocalautos}, surjective local automorphisms basically
coincide with $R$-algebra automorphisms.
In this section, we want to discuss the possible existence of non-surjective
local automorphisms on $\Pi$. We will see that this problem relates to a specific
large cardinal number $\mu_R$, cf. Definition \ref{def:muR}.

We will start discussing the central ultrafilter construction for the special case
of finite indecomposable rings $R$. This will need a small auxiliary result.

\begin{lemma} \label{lem:sigmagivesrho}
Let $\omega=X \,\dot\cup\, Y$ be a partition of $\omega$ such that
$Y$ is infinite, and let $\sigma: \omega \rightarrow \omega \setminus \{0\}$
be a bijection.
Then there exists a permutation $\rho$ of $\omega$ such that
${\rho\upharpoonright}_X={\sigma\upharpoonright}_X$ and $\rho(Y)=\sigma
(Y)\,\dot\cup\,\{0\}$.
\end{lemma}

\begin{proof}
Note that $\omega=\sigma(X)\,\dot\cup\,\sigma(Y)\,\dot\cup\,\{0\}$
is a partition of $\omega$. Let $Y=\{y_{i}:i\in\omega\}$ be an
enumeration of the elements of $Y$ and define
\[
\rho(x) = \left\{
  \begin{array}{ll}
     \sigma(x) & \mbox{for $x \in X$}, \\
     0 & \mbox{for $x = y_0$},\\
     \sigma(y_{i-1}) & \mbox{for $x=y_i$, $1\leq i\in\omega$}.
  \end{array} \right.
\]
Note that $\rho(X) = \sigma(X)$ and $\rho(Y) = \sigma(Y) \cup \{0\}$,
and thus $\rho$ is surjective. Moreover, $\rho\upharpoonright X$ and
$\rho\upharpoonright Y$ are injective with
\[
\rho(X) \cap \rho(Y) = \sigma(X) \cap \left( \sigma(Y) \cup \{0\} \right) = \emptyset,
\]
which shows that $\rho$ is bijective.
\end{proof}

With this we are ready for our first ultrafilter construction of a non-surjective local
automorphism.

\begin{lemma} \label{lem:finiteRnonsurjective}
Let $R$ be a finite indecomposable ring and $\Pi=\prod_{i\in\omega}Re_{i}$.
Then there exists some unital $\eta\in \operatorname{End}(\Pi)$ such that $\eta$ is an
$n$-local automorphism for all $n>0$, but not surjective. In particular,
$\eta$ is not an $R$-algebra automorphism of $\Pi$.
\end{lemma}

\begin{proof}
Let $\mathcal{U}$ be a non-principal ultrafilter for the set $\omega$ that contains all
cofinite subsets of $\omega$. Let $a\in\Pi$. Then there exist disjoint sets
$[a]_r\subseteq\omega$, $r\in R$, such that $a=\sum_{r\in R}re_{[a]_r}$.
Since $\mathcal{U}$ is an ultrafilter, there exists exactly one
$s=:\varphi(a)\in R$ such that $u(a):=[a]_{s}\in \mathcal{U}$. It is easy to check that
$\varphi:\Pi\rightarrow R$ is an $R$-algebra homomorphism. Now define
$\eta:\Pi\rightarrow\Pi$ by
\[
\eta(a)=\varphi(a)e_{0}+\sum_{i\in \omega} a_{i}e_{i+1}
\]
for all $a=\sum_{i\in\omega}a_{i}e_{i} \in\Pi$.
Then $\eta$ is an injective unital $R$-algebra homomorphism, but not surjective since,
for example, $e_0$ has no preimage under $\eta$.

Now let $a\in\Pi$. Note that $u(a)\in \mathcal{U}$ is an infinite set.
Let $Y$ be any infinite subset of $u(a)$, and let $X=\omega\setminus Y$.
Apply Lemma \ref{lem:sigmagivesrho} to $\omega=X\,\dot\cup\,Y$ and the
\emph{shift map} $\sigma$ on $\omega$, where $\sigma(i)=i+1$ for all $i\in\omega$.
This yields a permutation $\rho$ of $\omega$ with associated $R$-algebra automorphism
$\widehat{\rho} \in \operatorname{Aut}(\Pi)$. Note that
\begin{align*}
\widehat{\rho}(a) &= \widehat{\rho}\left( \sum_{x\in X} a_{x}e_{x}+\sum_{y\in Y} a_ye_{y} \right)
= \widehat{\rho}\left( \sum_{x\in X} a_{x}e_{x} \right) + \widehat{\rho}\left( \sum_{y\in Y} \varphi(a)e_{y} \right)\\
&= \widehat{\rho}\left( \sum_{x\in X} a_{x}e_{x} \right) + \widehat{\rho}\left( \varphi(a)e_Y \right)
= \sum_{x\in X}a_{x}e_{\rho(x)}+ \varphi(a) e_{\rho(Y)}\\
&= \sum_{x\in X}a_{x}e_{\sigma(x)}+ \varphi(a) e_{\sigma(Y)\cup \{0\}}
= \varphi(a)e_0 + \left( \sum_{x\in X}a_{x}e_{\sigma(x)}+ \varphi(a) e_{\sigma(Y)} \right)\\
&= \varphi(a)e_0 + \widehat{\sigma}(a) = \eta(a).
\end{align*}
This shows that $\eta$ is a local automorphism.

Now let $a_{1},a_{2},\ldots ,a_{n}\in\Pi$ and consider the infinite set
$Y=\bigcap_{1\leq i\leq n}u(a_{i}) \in \mathcal{U}$. Note that the automorphism
$\widehat{\rho}$ constructed above will have the property that
$\eta (a_{i})=\widehat{\rho}(a_{i})$ for all $1\leq i\leq n$.
Thus $\eta$ is an $n$-local automorphism.
\end{proof}

We include the following immediate consequence of Lemma \ref{lem:finiteRnonsurjective}.

\begin{corollary} \label{cor:finiteRnonsurjective}
If $R$ is a finite indecomposable ring and $P$ an infinite set, then there exists an injective
unital $R$-algebra endomorphism $\eta$ of $\Pi=\prod_{x\in P}Re_x$ such that $\eta\neq\widehat{\rho}$
for all injective maps $\rho:P\rightarrow P$.
\end{corollary}

\begin{proof}
Let $R$ be finite, $P$ infinite, and $\{x_\alpha: \alpha < |P|\}$ an enumeration of $P$.
Let the map $\varphi:\Pi\rightarrow R$ be
defined via an ultrafilter on $P$ as in Lemma~\ref{lem:finiteRnonsurjective}, and define
\[
\eta(a)=\varphi (a)e_{x_0}+\sum_{\alpha\, <\, |P|}a_{x_\alpha}e_{x_{\alpha+1}}
\]
for all $a=\sum_{x\in P}a_{x}e_{x} \in\Pi$.
Then $\eta$ is an injective unital $R$-algebra endomorphism of $\Pi$. Moreover, $e_{x_0}\notin \eta(\Pi)$,
but $\eta$ followed by the natural projection onto the $e_{x_0}$-coordinate is non-zero.
This shows that $\eta\neq\widehat{\rho}$ for all maps $\rho:P\rightarrow P$.
\end{proof}

The following example illustrates an interesting variation of the construction provided in the proof of
Lemma \ref{lem:finiteRnonsurjective}.

\begin{example}
Let $R$ be a commutative ring with identity and consider the $R$-subalgebra
\[
A=\left<\, \sum_{i\in \omega} e_i,\, \bigoplus_{i\in \omega}Re_{i} \right>
\]
of $\Pi=\prod_{i\in \omega} Re_{i}$. Then each $a\in A$ has the form
\[
a=\sum_{i=0}^{n(a)} s_{i}e_{i}+\sum_{i\in \omega}a^{\ast}e_{i}
\]
for $s_{i}\in R$ and a unique $a^{\ast}\in R$. Now define an
$R$-linear map $\eta:A\rightarrow A$ by
\[
\eta(a)=a^{\ast}e_{0}+\sum_{i=0}^{n(a)} s_{i}e_{i+1}+\sum_{i\in \omega}a^{\ast}e_{i+1}
\]
for all $a\in A$. Note that $e_0$ is not in the image of $\eta$, and thus $\eta$ is not surjective. Let
$n>0$. We will show that $\eta$ is an $n$-local automorphism of
$A$. To this end, let $a_j\in A$ for $1\leq j\leq n$ and pick some integer $m$ with
$m>n(a_j)$ for all $1\leq j\leq n$. Define a
permutation $\rho:\omega\rightarrow\omega$ by $\rho(i)=i+1$ for all $m\neq
i<\omega$ and $\rho(m)=0$. Then $\rho$ induces an $R$-algebra automorphism
$\widehat{\rho}$ of the $R$-algebra $A$ and it is easy to check that $\eta(a_j)
=\widehat{\rho}(a_j)$ for all $1\leq j\leq n$, which shows that $\eta$ is
an $n$-local automorphism of $A$ which is not surjective and thus not an
$R$-algebra automorphism of $A$.

For fields $R$, we note another quite remarkable property of the local automorphism~$\eta$: Let $n>0$. Then $\eta^n(A) \cong A$.
Moreover, any $R$-algebra $\eta^n(A) \subseteq B \subseteq A$ is isomorphic to $A$, and there exist precisely
$2^n$ such intermediate $R$-algebras $B$.
\end{example}

Next we want to consider indecomposable rings $R$ of arbitrary size. This will need the following definition.

\begin{definition} \label{def:muR}
Let
\[
\mu_R = \left\{
  \begin{array}{ll}
     \aleph_0, & \mbox{\quad if $R$ is finite}, \\
     \mbox{the smallest cardinal with a nontrivial}\\
     \mbox{$|R|^+$-additive $0,1$\mbox-valued measure}, & \mbox{\quad if $R$ is infinite}.
  \end{array} \right.
\]
For $R$ infinite, we will set $\mu_R = \infty$ in case that such a cardinal does not exist.
\end{definition}

\begin{remark} \label{rem:muRprop}\mbox{}
\begin{itemize}
\item[(1)] For infinite $R$, we will briefly discuss the cardinal condition on $\mu_R$. First,
note that $\mu_R > |R|$ by definition, and that $\mu_R$ must allow a
nontrivial $\sigma$-additive $0,1$\mbox-valued measure. Let $\mu_0$ denote the smallest
such cardinal. It is well-known that this cardinal $\mu_0$ is measurable and inaccessible,
cf. \cite[Chapter~10]{Jech}. Thus, the existence of $\mu_R$ implies the existence of
measurable cardinals in the chosen model of set theory, and it is consistent with {\rm ZFC}
that no measurable cardinals may exist. In this case $\mu_R = \infty$ will apply, and
statements with respect to cardinals $\ge \mu_R$ will become void.
\item[(2)] Note, that $\mu: \mathcal{P}(P) \rightarrow \{0,1\}$
is a nontrivial $|R|^+$-additive measure on a set $P$ if and only if
\[
\mathcal{U}=\{X\subseteq P:\mu(X)=1\}
\]
is a non-principal $|R|^+$-complete ultrafilter on $P$, cf. \cite[Chapter~10]{Jech}.
Thus, in the following, all measure arguments are synonymous to ultrafilter arguments.
\end{itemize}
\end{remark}

We note another immediate consequence of $\mu_R$.

\begin{proposition} \label{prop:muRnontrivmeasure}
For any sets $R$ and $P$ the following are equivalent:
\begin{itemize}
\item[(1)] There exists a nontrivial $|R|^+$-additive $0,1$\mbox-valued measure on $P$.
\item[(2)] $|P| \ge \mu_R$.
\end{itemize}
\end{proposition}

\begin{proof}
For $R$ finite, it is well-known that $P$ allows a non-principal ultrafilter if and only
if $|P| \ge \aleph_0$. For $R$ infinite, (1) evidently implies (2).

Let now $R$ be infinite and $|P| \ge \mu_R$. Choose a subset $P'\subseteq P$ of cardinality
$|P'|=\mu_R$ and a nontrivial $|R|^+$-additive measure $\mu': \mathcal{P}(P') \rightarrow \{0,1\}$
on $P'$. Then $\mu(X)=\mu'(P'\cap X)$ for $X\subseteq P$ defines a nontrivial $|R|^+$-additive
$0,1$\mbox-valued measure $\mu$ on $P$.
\end{proof}

We are now ready for the main result of this section.

\begin{theorem} \label{thm:maincartprod}
For any indecomposable ring $R$ and $\Pi=\prod_{x\in P}Re_{x}$ the following are equivalent:
\begin{itemize}
\item[(1)] $\operatorname{LAut}(\Pi) =  \operatorname{Aut}(\Pi)$.
\item[(2)] $|P| < \mu_R$.
\end{itemize}
\end{theorem}

This theorem will be the collective achievement of a series of intermediate results.
First, we would like to show that the existence of a nontrivial local automorphism
implies $|P| \ge \mu_R$. The following auxiliary result will be crucial.

\begin{lemma} \label{lem:gammaspecialprop}
Let $R$ be an infinite indecomposable ring and
$\Pi=\prod_{x\in P}Re_{x}$. Let $\gamma :\Pi\rightarrow R$ be some $R$-linear map such that
\begin{align} \label{eq:gammaspecialprop}
\gamma(a)\in\{a_{x}:x\in P\} \mbox{ for all } a =\sum_{x\in P} a_{x}e_{x}\in \Pi \mbox{ and }
\gamma(e_{x})=0 \mbox{ for all } x\in P.
\end{align}
Then $|P| \geq\mu_R$.
\end{lemma}

\begin{proof}
For $X\subseteq P$ define $e_{X}=\sum_{x\in X}e_{x}$. Then
$e_{\emptyset}=0$, and $e=e_{P}$ is the identity element of $\Pi$. Note that
by \eqref{eq:gammaspecialprop} we have $\gamma(e_{X})\in\{0,1\}$ for all $X\subseteq P$.
We will identify $0,1 \in R$ with their real-valued conterparts,
and define $\mu: \mathcal{P}(P) \rightarrow \{0,1\}$ by $\mu(X) = \gamma(e_X)$ for all $X\subseteq P$.
\medskip

To prove $|P| \geq\mu_R$, we will show that $\mu$ is a nontrivial $|R|^+$-additive measure on $P$:\\
Obviously, by \eqref{eq:gammaspecialprop},
\begin{align} \label{eq:muismeasure1}
\mu(\emptyset)=\gamma(0)=0,\: \mu(P)=\gamma(e)=1, \mbox{ and } \mu( \{x\}) =\gamma(e_x) =0 \mbox{\quad for all } x\in P.
\end{align}
Let $X\subseteq P$. Then
\[
1=\gamma(e)=\gamma(e_{X}+e_{P\setminus X})= \gamma(e_{X})+\gamma(e_{P\setminus X})= \mu(X)+\mu(P\setminus X)
\]
Thus
\begin{align} \label{eq:muismeasure2}
\mu(P\setminus X) = 1-\mu(X) \mbox{\quad for all } X\subseteq P.
\end{align}
We claim that
\begin{align} \label{eq:muismeasure3}
\mu(X) = \mu(Y) =1 \mbox{ for } X,Y\subseteq P \mbox{ implies } X\cap Y \not= \emptyset.
\end{align}
Assume $\mu(X) = \mu(Y) =1$ and $X\cap Y = \emptyset$.
Choose $f\in R\setminus \{0\}$ and $g\in R\setminus \{0,-f\}$. Then
\[
f+g=f\mu(X)+g\mu(Y)=f\gamma(e_{X})+g\gamma(e_{Y})=\gamma(fe_{X}+ge_{Y})\in\{0,f,g\}
\]
by \eqref{eq:gammaspecialprop}, and thus
$f+g\in\{0,f,g\}$, a contradiction to our choice of $f$ and $g$.\\
We next claim that
\begin{align} \label{eq:muismeasure4}
\mu(X) \le \mu(Y) \mbox{\quad for all } X\subseteq Y\subseteq P.
\end{align}
Assume $\mu(X) > \mu(Y)$, thus $\mu(X)=1$ and $\mu(Y)=0$. With \eqref{eq:muismeasure2}
we have $\mu(P\setminus Y)=1=\mu(X)$ with $(P\setminus Y)\cap X= \emptyset$,
a contradiction to \eqref{eq:muismeasure3}.

So far, we have shown that $\mu$ is a nontrivial measure on $P$.
It remains to show that $\mu$ is $|R|^+$-additive:

To this end, let $\kappa \le |R|$ be a cardinal and $X_\alpha \subseteq P$
for $\alpha < \kappa$ be mutually disjoint sets with $X=\bigcup_{\alpha < \kappa} X_\alpha$,
and assume that $\mu(X)=1$ and $\mu(X_\alpha)=0$ for all $\alpha < \kappa$.
Set $X_\kappa = P \setminus X$, and observe that $\mu(X_\kappa) = 0$ by \eqref{eq:muismeasure2}.
Thus, we have a partition
\begin{align} \label{eq:muismeasure5}
P=\bigcup_{\alpha \le \kappa} X_\alpha \mbox{\quad with \quad} \mu(X_\alpha)=0 \mbox{ for all } \alpha \le \kappa.
\end{align}
Pick $a_\alpha \in R$ such that
$0\neq a_\alpha \neq a_\beta$ for all $\alpha \neq \beta \le \kappa$ and let
$a=\sum_{\alpha \le \kappa}a_\alpha e_{X_\alpha}$. By \eqref{eq:gammaspecialprop} there is a
unique $\beta$ such that $\gamma(a)=a_\beta$. Now define
\[
b=a -a_\beta e_{X_\beta} =\sum_{\substack{\alpha \le \kappa\\ \alpha \not= \beta}}a_\alpha e_{X_\alpha}.
\]
Then $\gamma(b)\in\{0, a_\alpha:\beta \neq \alpha \le \kappa \}$ with \eqref{eq:gammaspecialprop}, but
\[
\gamma(b)= \gamma(a -a_\beta e_{X_\beta}) = \gamma(a)- a_\beta \gamma(e_{X_\beta})
= \gamma(a)- a_\beta \mu(X_\beta) = \gamma(a) = a_\beta
\]
with \eqref{eq:muismeasure5}, a contradiction to our choice of $a_\beta$.
\end{proof}

We are all set for proving that Property (2) implies Property (1) in Theorem~\ref{thm:maincartprod}.

\begin{lemma} \label{lem:maincartprod1}
Let $R$ be an indecomposable ring such that $\operatorname{LAut}(\Pi) \not= \operatorname{Aut}(\Pi)$
for $\Pi=\prod_{x\in P}Re_{x}$. Then $|P| \ge \mu_R$.
\end{lemma}

\begin{proof}
The inclusion $\operatorname{Aut}(\Pi) \subseteq \operatorname{LAut}(\Pi)$ is obvious.
Thus $\operatorname{LAut}(\Pi) \not= \operatorname{Aut}(\Pi)$ implies the existence
of some $\eta \in \operatorname{LAut}(\Pi) \setminus \operatorname{Aut}(\Pi)$.

Let $E=\{e_{x}:x\in P\}$. Then $E$ is the set of all primitive idempotents of the
$R$-algebra $\Pi$ and thus $\eta(E)\subseteq E$. We infer that there exists a
map $\rho:P\rightarrow P$ such that $\eta(e_{x})=e_{\rho(x)}$ for all $x\in P$.
Since $\eta$ is injective, the map $\rho$ is injective.

Suppose, for the moment, that $P$ is finite. Then for all
$a =\sum_{x\in P} a_{x}e_{x}\in \Pi$ holds
\[
\eta(a) =\eta\left( \sum_{x\in P} a_{x}e_{x} \right) = \sum_{x\in P} \eta(a_{x}e_{x})
= \sum_{x\in P} a_{x}\eta(e_{x}) = \sum_{x\in P} a_{x}e_{\rho(x)}= \widehat\rho(a),
\]
thus $\eta = \widehat\rho$.
Furthermore, $E$ is finite, and $\eta(E)\subseteq E$ implies $\eta(E)= E$. Thus $\rho$ is a
permutation of $P$, and $\eta = \widehat\rho \in \operatorname{Aut}(\Pi)$ with
Proposition \ref{prop:inducedauto}, a contradiction.

Therefore, $P$ is infinite. For $R$ finite, this implies $|P| \ge \aleph_0= \mu_R$ and the
proof is complete. Hence, let $R$ be infinite.

Suppose, for the moment, that $\rho$ is surjective. Then $\eta=\widehat \rho \in \operatorname{Aut}(\Pi)$
follows as in Theorem \ref{thm:surjectivelautisaut}, a contradiction.

Therefore, $\rho$ is non-surjective, and we may pick some $y \in P\setminus \operatorname{Im}(\rho)$.
Let $\pi:\Pi\rightarrow Re_y\cong R$ be the natural projection onto the $e_y$-coordinate.
Let $\gamma=\pi\circ \eta$. Since $\eta$ is a local automorphism, there exists
for each $a =\sum_{x\in P} a_{x}e_{x}\in \Pi$ some permutation
$\rho_a$ such that $\eta(a)=\widehat\rho_a(a)=\sum_{x\in P}a_xe_{\rho_a(x)}$. Thus
\begin{align} \label{eq:gammaspecialprop1}
\gamma(a)=a_{\rho_a^{-1}(y)}\in \{a_{x}:x\in P\}.
\end{align}
Furthermore, for all  $x \in P$
\begin{align} \label{eq:gammaspecialprop2}
\gamma(e_x)=\pi(\eta(e_x)) =\pi(e_{\rho(x)})=0
\end{align}
since $y\notin\operatorname{Im}(\rho)$. With \eqref{eq:gammaspecialprop1} and
\eqref{eq:gammaspecialprop2}, the $R$-linear map $\gamma$ has Property \eqref{eq:gammaspecialprop},
and Lemma \ref{lem:gammaspecialprop} yields $|P| \ge \mu_R$.
\end{proof}

Next follows another generalization of the ultrafilter construction of Lemma \ref{lem:finiteRnonsurjective}.

\begin{theorem} \label{thm:maincartprod2}
Let $R$ be an indecomposable ring and $P$ be a set with $|P| \ge \mu_R$.
Then for $\Pi=\prod_{x\in P}Re_{x}$ there exists some unital
$\eta\in \operatorname{End}(\Pi)$ such that $\eta$ is an $n$-local
automorphism for all $n>0$, but not surjective. In particular,
$\eta$ is not an $R$-algebra automorphism of $\Pi$.
\end{theorem}

\begin{proof}
The case of finite $R$ is covered by Lemma \ref{lem:finiteRnonsurjective} and Corollary \ref{cor:finiteRnonsurjective}.
Thus let $R$ be infinite. As $|P| \ge \mu_R$, we may choose a nontrivial $|R|^+$-additive measure
$\mu: \mathcal{P}(P) \rightarrow \{0,1\}$, cf. Proposition \ref{prop:muRnontrivmeasure}.
Let $a\in\Pi$. Then there exist disjoint sets
$[a]_r\subseteq P$, $r\in R$, such that $a=\sum_{r\in R}re_{[a]_r}$.
Since $\mu$ is $|R|^+$-additive, there exists exactly one $s=:\varphi(a)\in R$
such that $\mu([a]_{s})=1$. It is easy to check that $\varphi:\Pi\rightarrow R$
is an $R$-algebra homomorphism. With respect to some enumeration
$\{x_\alpha: \alpha < |P|\}$ of $P$ define $\eta:\Pi\rightarrow\Pi$ by
\[
\eta(a)=\varphi (a)e_{x_0}+\sum_{\alpha\, <\, |P|}a_{x_\alpha}e_{x_{\alpha+1}}.
\]
for all $a=\sum_{x\in P}a_{x}e_{x} \in\Pi$.
As in Lemma \ref{lem:finiteRnonsurjective} and Corollary \ref{cor:finiteRnonsurjective},
we can show that $\eta$ has all the necessary properties. Note, in particular, that
\[
[a]_{\varphi(a)} \in \mathcal{U} = \{X\subseteq P: \mu(X)=1\},
\]
where $\mathcal{U}$ is a non-principal $|R|^+$-complete ultrafilter on $P$, cf. Remark \ref{rem:muRprop}\,(2).
\end{proof}

We mention the following obvious generalization of Corollary \ref{cor:finiteRnonsurjective}.

\begin{corollary} \label{cor:Rnonsurjective}
If $R$ is an indecomposable ring and $P$ a set with $|P| \ge \mu_R$, then there exists an injective
unital $R$-algebra endomorphism $\eta$ of $\Pi=\prod_{x\in P}Re_x$ such that $\eta\neq\widehat{\rho}$
for all injective maps $\rho:P\rightarrow P$.
\end{corollary}

Theorem \ref{thm:maincartprod} now follows from Lemma \ref{lem:maincartprod1}
and Theorem \ref{thm:maincartprod2}.

\subsection{$R$-algebra endomorphisms}

In this section we will revisit and generalize the idea that maps $\rho: P\rightarrow P$ induce
$R$-algebra endomorphisms $\widehat \rho \in \operatorname{End}(\Pi)$. Our reward will be a
characterization of $R$-algebra endomorphisms in terms of $\mu_R$, cf. Theorem
\ref{thm:mainendo}.

We start with a simple definition. As motivation, consider an arbitrary $R$-algebra
endomorphism $\eta\in \operatorname{End}(\Pi)$.
For each $x\in P$, the element $\eta(e_{x})$ is an idempotent, and there exists some
subset $A_{x}\subseteq P$ such that
\[
\eta(e_{x})=e_{A_{x}}.
\]
Note, that this includes the
possible options $A_{x}=\emptyset$ for $\eta(e_{x})=0$ and $A_{x}=P$ for $\eta(e_{x})=e$,
the identity element of $\Pi$. Furthermore, for $x,y \in P$ with $x\not= y$ we have
\begin{align} \label{eq:Axdef}
e_{A_x \cap A_y} = e_{A_x} e_{A_y} = \eta(e_{x})\eta(e_{y})=\eta(e_{x} e_{y}) = \eta(0)=0,
\end{align}
hence $A_x \cap A_y=\emptyset$. It follows that $\{A_{x}:x\in P\}$ is a
family of pairwise disjoint subsets of $P$.

\begin{definition}
Let $\eta\in \operatorname{End}(\Pi)$. We call $\eta$ \emph{induced}, if
there exists a family $\mathcal{F}=\{A_{x}:x\in P\}$ of pairwise disjoint
subsets of $P$ such that
\begin{align} \label{eq:inducedeq}
\eta(a)=\sum_{x\in P} a_{x}e_{A_{x}} \mbox{\quad for all }
a=\sum_{x\in P}a_{x}e_{x} \in \Pi.
\end{align}
In this case, we will call \emph{$\eta=\eta_{\mathcal{F}}$ induced by $\mathcal{F}$}.
\end{definition}

It is readily verified that Equation \eqref{eq:inducedeq} indeed defines an $R$-algebra
homomorphism for every family $\mathcal{F}=\{A_{x}:x\in P\}$ of pairwise disjoint subsets of $P$.
The induced homomorphism $\eta_\mathcal{F}$ is injective if and only if
$A_x \not= \emptyset$ for all $x\in P$. Note also, that for every injective map $\rho: P \rightarrow P$
the $R$-algebra homomorphism $\widehat \rho$ is induced by the family
\[
\mathcal{F}=\{A_{x}:x\in P\} \mbox{\quad with } A_x=\{\rho(x)\}.
\]
Finally, for every induced $R$-algebra homomorphism $\eta$ we can uniquely recover the inducing family
$\mathcal{F}=\{A_{x}:x\in P\}$ from Equation \eqref{eq:Axdef}.

The following propositions list some more results on induced homomorphisms.

\begin{proposition} \label{prop:inducedhomo1}
Let $\eta\in \operatorname{End}(\Pi)$ with $\eta(e_{x})=e_{A_{x}}$ for all $x \in P$.
If $P=\bigcup_{x\in P}A_{x}$, then $\eta$ is induced by $\mathcal{F}=\{A_{x}:x\in P\}$.
\end{proposition}

\begin{proof}
Let $a=\sum_{x\in P}a_{x}e_{x} \in \Pi$. Then $e=\sum_{x\in P}e_{A_{x}}$ and
$\eta(a)=\sum_{x\in P}\eta(a)e_{A_{x}}$. Moreover,
\[
\eta(a) e_{A_x}=\eta(a)\eta(e_{x})=\eta(ae_{x})=\eta(a_xe_{x})=a_x\eta(e_{x})=a_{x}e_{A_{x}},
\]
and it follows that
$\eta(a)=\sum_{x\in P}a_{x}e_{A_{x}}$. This shows that $\eta$ is induced  by $\mathcal{F}$.
\end{proof}

More generally still, the following holds.

\begin{proposition} \label{prop:inducedhomo2}
Let $\eta\in \operatorname{End}(\Pi)$ with $\eta(e_{x})=e_{A_{x}}$ for all $x \in P$,
let $A=\bigcup_{x\in P}A_{x}$, $B = P\setminus A$ and $\Pi=\Pi_{A}\oplus\Pi_{B}$ with
$\Pi_A=\prod_{x\in A}Re_{x}$ and $\Pi_B=\prod_{x\in B}Re_{x}$.
Then for  $\eta=\eta_{A}\oplus\eta_{B}$ with $\eta_A, \eta_B \in \operatorname{End}(\Pi)$
the $\Pi_A$- and $\Pi_B$-component of $\eta$, respectively, the following holds.
\begin{itemize}
\item[(a)] $\eta_A$ is induced by $\mathcal{F}=\{A_{x}:x\in P\}$.
\item[(b)] $\eta_{B}(e_{x})=0$ for all $x\in P$.
\item[(c)] $\eta$ is induced if and only if $\eta_B=0$. In this case, $\eta =\eta_A = \eta_{\mathcal{F}}$.
\end{itemize}
\end{proposition}

We can strengthen the result of Corollary \ref{cor:Rnonsurjective} as follows.

\begin{corollary} \label{cor:Rnonsurjective2}
If $R$ is an indecomposable ring and $P$ a set with $|P| \ge \mu_R$, then there exists an injective
unital $R$-algebra endomorphism $\eta$ of $\Pi=\prod_{x\in P}Re_x$ such that $\eta$ is not induced.
\end{corollary}

\begin{proof}
Let $\{x_\alpha: \alpha < |P|\}$ be some enumeration of $P$.
Let the map $\varphi:\Pi\rightarrow R$ be
defined via a non-principal $|R|^+$-complete ultrafilter on $P$ as in Theorem \ref{thm:maincartprod2},
and define
\[
\eta(a)=\varphi (a)e_{x_0}+\sum_{\alpha\, <\, |P|}a_{x_\alpha}e_{x_{\alpha+1}}
\]
for all $a=\sum_{x\in P}a_{x}e_{x} \in\Pi$.
Then $\eta$ is an injective unital $R$-algebra endomorphism of $\Pi$.
Moreover, $\eta(e_{x_\alpha})= e_{x_{\alpha+1}}$ and $A_{x_\alpha} = \{x_{\alpha+1}\}$ for all $\alpha < |P|$.
As $x_0 \notin \bigcup_{x \in P} A_x$, but $\eta$ followed by the natural projection onto the
$e_{x_0}$-coordinate is non-zero, $\eta$ is not induced.
\end{proof}

The central result of this section will be the following adaption of Theorem \ref{thm:maincartprod}
to the situation of $R$-algebra endomorphisms of $\Pi$. We will require $R$ to be a field.

\begin{theorem} \label{thm:mainendo}
For any field $R$ and $\Pi=\prod_{x\in P}Re_{x}$ the following are equivalent:
\begin{itemize}
\item[(1)] Every $\eta \in \operatorname{End}(\Pi)$ is induced.
\item[(2)] $|P| < \mu_R$.
\end{itemize}
\end{theorem}

\begin{proof}
Based on Corollary \ref{cor:Rnonsurjective2} we just need to check that (2) implies (1).
Let us therefore assume that there exists some $\eta \in \operatorname{End}(\Pi)$ which
is not induced.

Decompose $\eta=\eta_{A}\oplus\eta_{B}$ as in Proposition \ref{prop:inducedhomo2}.
As $\eta$ is not induced, we have $\eta_B \not= 0$. In particular,
we may pick some $y \in B = P\setminus A$ such that $\gamma:=\pi\circ \eta \not= 0$,
where $\pi:\Pi\rightarrow Re_y\cong R$ is the natural projection onto the $e_y$-coordinate.
Note that $\gamma$ is an $R$-algebra homomorphism.

With Proposition \ref{prop:inducedhomo2}(b),
\begin{align} \label{eq:gammaspecialprop3}
\gamma(e_x)=\pi(\eta(e_x)) =\pi(\eta_B(e_x))=0 \mbox{\quad for all } x \in P.
\end{align}
We claim
\begin{align} \label{eq:gammaspecialprop4}
\gamma(e)=1.
\end{align}
Assume $\gamma(e)\not= 1$. In this case, we have $\gamma(e)=0$, the only other idempotent in $R$.
Hence,
\[
\gamma(a)=\gamma(ea)=\gamma(e)\gamma(a)=0
\]
for all $a\in \Pi$, a contradiction to our choice of $y$.

We next claim
\begin{align} \label{eq:gammaspecialprop5}
\gamma(a)\in\{a_{x}:x\in P\} \mbox{\quad for all } a =\sum_{x\in P} a_{x}e_{x}\in \Pi.
\end{align}
To see this, consider the element $b=a-\gamma(a)e$. We have
\[
\gamma(b)=\gamma(a-\gamma(a)e)=\gamma(a)-\gamma(\gamma(a)e) = \gamma(a)-\gamma(a)\gamma(e) = \gamma(a)-\gamma(a) = 0.
\]
If $b \in \Pi$ were invertible, this would yield
\[
1= \gamma(e)=\gamma(b)\gamma(b^{-1}) = 0,
\]
a contradiction. Thus $b$ must not be invertible. Hence, as $R$ is a field, one coordinate of $b=a-\gamma(a)e$ must be $0$,
and thus $\gamma(a)\in\{a_{x}:x\in P\}$.

With \eqref{eq:gammaspecialprop3} and \eqref{eq:gammaspecialprop5}, the $R$-linear map
$\gamma$ has Property \eqref{eq:gammaspecialprop}, and Lemma \ref{lem:gammaspecialprop} yields $|P| \ge \mu_R$.
\end{proof}

\section{Local automorphisms of finitary incidence algebras} \label{sec:fips}

In the following, we will investigate the local automorphisms of finitary incidence algebras.
Let us first fix some basic notation for the remainder of this section.

\begin{notation}
Let $(P,\le)$ be a poset, and let $R$ be a commutative, indecomposable ring.
Then $FI(P)$ will denote the induced finitary incidence algebra. We will associate
any element $a \in FI(P)$ with its standard representation $a=\sum_{x \le y} a_{xy}e_{xy}$.
For convenience we may write $e_x$ instead of $e_{xx}$, and $a_x$ instead of $a_{xx}$.
\end{notation}

We will also need the following definition.

\begin{definition}
Let
\[
Z=Z(FI(P)) = \{a: a_x =0 \mbox{ for all } x\in P \}\subseteq FI(P).
\]
\end{definition}

We gather some of the basic properties of $Z(FI(P))$ for later use.

\begin{proposition} For $Z=Z(FI(P))$ the following holds. \label{prop:propZ}
\begin{itemize}
\item[(a)] $Z\vartriangleleft FI(P)$ is a two-sided ideal.
\item[(b)] $FI(P)/Z \cong \prod_{x\in P} Re_x$.
\item[(c)] $\varphi(Z) = Z$ for all $\varphi \in \operatorname{Aut}(FI(P))$.
\end{itemize}
\end{proposition}

\begin{proof}
Parts (a) and (b) are easy to check. For (c), we just need to check $\varphi(Z) \subseteq Z$,
as $Z \subseteq \varphi(Z)$ follows from $\varphi^{-1}(Z) \subseteq Z$. Note, however, that
$\varphi = \psi_f \circ M_\sigma \circ \widehat \rho$ with Theorem \ref{thm:charaut}, and that
$\psi_f(Z) \subseteq Z$, $M_\sigma(Z) \subseteq Z$ and $\widehat \rho(Z) \subseteq Z$
is immediate.
\end{proof}

Our main result will be the following generalization of Theorem \ref{thm:maincartprod}.

\begin{theorem} \label{thm:mainlautoffip}
Let $(P,\le)$ be a poset and $R$ an indecomposable ring.
\begin{itemize}
\item[(a)] For $|P| < \mu_R$, we have $\operatorname{LAut}(FI(P)) = \operatorname{Aut}(FI(P))$.
\item[(b)] For $|P| \ge \mu_R$, non-surjective local automorphisms may exist.
\end{itemize}
\end{theorem}

Basically the same arguments apply to provide a corresponding generalization of Theorem \ref{thm:surjectivelautisaut}.

\begin{theorem}  \label{thm:mainlautoffip2}
Let $(P,\le)$ be a poset, $R$ an indecomposable ring with $|R|\ge 3$, and $\eta$
a local automorphism of $FI(P)$ such that
\[
e_{x}\in\: \operatorname{Im}(\eta)+Z(FI(P)) \mbox{\quad for all } x\in P.
\]
Then $\eta$ is an $R$-algebra automorphism.
\end{theorem}

This includes as an important special case the following result.

\begin{corollary} \label{cor:mainlautoffip2}
Let $(P,\le)$ be a poset and $R$ an indecomposable ring with $|R|\ge 3$.
Then every surjective local automorphism of $FI(P)$ is an $R$-algebra automorphism.
\end{corollary}

For part (b) of Theorem \ref{thm:mainlautoffip}, we simply refer to Theorem \ref{thm:maincartprod2}
for a counterexample. Thus we only need to prove part (a), which will be the ultimate goal
of a very elaborate chain of intermediate results. As a general agenda, we will try to mimic the
proof of Theorem \ref{thm:charaut}. Hence, confronted with an arbitrary
$\eta \in \operatorname{LAut}(FI(P))$ we will attempt to split off suitable canonical automorphisms
$\widehat \rho$, $\psi_f$, and $M_\sigma$, showing that the remaining local automorphism is the
identity map on $FI(P)$.

\subsection{Step 1: Splitting off $\widehat \rho$}\mbox{}\medskip

We will use Theorem \ref{thm:maincartprod} to isolate a hopeful candidate $\rho:P \rightarrow P$ for a suitable
order automorphism of $(P,\le)$. Here, the cardinal condition $|P| < \mu_R$ will become crucial in ascertaining that
$\rho$ is surjective.

\begin{proposition} \label{prop:inducedrho}
Let $|P| < \mu_R$ and $\eta \in \operatorname{LAut}(FI(P))$. Then there exists a permutation $\rho: P \rightarrow P$
such that
\begin{align} \label{eq:inducedrho}
\eta (a) \in \sum_{x\in P}a_xe_{\rho(x)} +Z(FI(P)) \mbox{\quad for all } a\in FI(P).
\end{align}
\end{proposition}

\begin{proof}
For any $R$-algebra automorphisms $\varphi \in \operatorname{Aut}(FI(P))$, we have $\varphi(Z)=Z$, and
$\varphi$ induces a canonical $R$-algebra automorphism $\overline \varphi$ on $FI(P)/Z \cong \prod_{x\in P} Re_x$.
As a consequence, we have $\eta(Z)\subseteq Z$, and
\[
\overline \eta (a+Z)=\eta(a)+Z \mbox{\quad for all } a\in FI(P)
\]
induces a canonical local automorphism $\overline \eta$ on $FI(P)/Z$. Applying
Theorem~\ref{thm:maincartprod} yields $\overline \eta \in \operatorname{Aut}(FI(P)/Z)$,
and with Proposition \ref{prop:inducedauto} there exists a permutation $\rho: P \rightarrow P$ with
\begin{align} \label{eq:inducedrho2}
\overline \eta \left(\sum_{x\in P}a_xe_x +Z\right) = \sum_{x\in P}a_xe_{\rho(x)} +Z \mbox{\quad for all } a\in FI(P).
\end{align}
Equation \eqref{eq:inducedrho} is now immediate.
\end{proof}

Replacing Theorem \ref{thm:maincartprod} in the last proof by Theorem \ref{thm:surjectivelautisaut}
leads to the following corollary as a starting point for the proof of Theorem \ref{thm:mainlautoffip2}.

\begin{corollary} \label{cor:inducedrho}
Let $|R|\ge 3$, and let $\eta \in \operatorname{LAut}(FI(P))$ with
$e_{x}\in\: \operatorname{Im}(\eta)+Z(FI(P))$ for all $x\in P$.
Then there exists a permutation $\rho: P \rightarrow P$ such that
\[
\eta (a) \in \sum_{x\in P}a_xe_{\rho(x)} +Z(FI(P)) \mbox{\quad for all } a\in FI(P).
\]
\end{corollary}

We need to show that the permutation $\rho:P\rightarrow P$ in Proposition \ref{prop:inducedrho} is
actually an order automorphism of $(P,\le)$. This will need a more detailed knowledge of the
structure of $\operatorname{Aut}(FI(P))$. Our arguments will transfer immediately to Corollary~\ref{cor:inducedrho}.

\begin{lemma} \label{lem:inducedrho1}
Let $|P| < \mu_R$, $|R|\ge 3$, and $\eta \in \operatorname{LAut}(FI(P))$. Then there exists an order automorphism
$\rho$ of $(P,\le)$ such that $\eta (a) \in \widehat \rho(a) +Z(FI(P))$ for all $a\in FI(P)$.
\end{lemma}

\begin{proof} We continue investigating the permutation $\rho:P\rightarrow P$ from the proof of
Proposition \ref{prop:inducedrho}.

For any $a \in FI(P)$, there exists some $\varphi_a \in \operatorname{Aut}(FI(P))$ with $\eta(a)=\varphi_a(a)$.
With Theorem~\ref{thm:charaut} we have $\varphi_a = \psi_{f_a} \circ M_{\sigma_a} \circ \widehat \rho_a$,
and
\[
\overline \eta(a+Z) = \overline{\varphi_a}(a+Z) = \left( \overline{\psi_{f_a}} \circ \overline{M_{\sigma_a}} \circ \overline{\widehat{\rho}_a} \right) (a+Z)
=\left( \overline{\psi_{f_a}} \circ \overline{M_{\sigma_a}} \right) \left(\sum_{x\in P}a_xe_{\rho_a(x)} +Z\right)
\]
holds for the induced maps on $FI(P)/Z$. Note, however, that $\psi_{f_a}$ and $M_{\sigma_a}$ induce
the identity map on $FI(P)/Z$. Thus,
\begin{align} \label{eq:inducedrho3a}
\overline \eta \left(\sum_{x\in P}a_xe_x +Z\right) = \sum_{x\in P}a_xe_{\rho_a(x)} +Z \mbox{\quad for all } a\in FI(P).
\end{align}
Comparing \eqref{eq:inducedrho2} and \eqref{eq:inducedrho3a} yields
\begin{align} \label{eq:inducedrho3}
\sum_{x\in P}a_xe_{\rho(x)} = \sum_{x\in P}a_xe_{\rho_a(x)} \mbox{\quad for all } a\in FI(P).
\end{align}

Let now $x,y \in P$ be arbitrary distinct elements, and choose $a_x, a_y \in R \setminus \{0\}$ with $a_x \not= a_y$.
Application of \eqref{eq:inducedrho3} to the element $a = a_xe_x+a_ye_y$ yields
\[
a_xe_{\rho(x)}+a_ye_{\rho(y)} = a_xe_{\rho_a(x)}+a_ye_{\rho_a(y)}.
\]
Thus, comparing coordinates, we have
\begin{align} \label{eq:inducedrho4}
\rho(x) = \rho_a(x) \mbox{\quad and \quad} \rho(y) = \rho_a(y).
\end{align}

Now, if $x < y$, then the order automorphism $\rho_a$ of $(P,\le)$ yields $\rho_a(x) < \rho_a(y)$, and
\[
\rho(x)= \rho_a(x) < \rho_a(y) = \rho(y).
\]
Similarly, $x> y$ yields $\rho(x)> \rho(y)$,
while $x,y$ incomparable yields $\rho(x),\rho(y)$ incomparable, and $\rho$ is an order automorphism.
\end{proof}

Once again, the situation $|R|= 2$ has to be treated as an exceptional case and will need some new ideas.
Note that for $|R|=2$ we have $|P| < \mu_R = \aleph_0$, and $(P,\le)$ is a finite poset. For any $x\in P$, let
$h(x)$ denote the \emph{height of $x$}, the size of a largest chain in $(P,\le)$ with maximal element $x$.
Thus, $h(x)=1$ if and only if $x$ is a minimal element of $(P,\le)$. Note, that
order automorphisms preserve heights.

The following lemma holds for indecomposable rings $R$ of  arbitrary size.

\begin{lemma} \label{lem:inducedrho2}
Let $P$ be finite, and $\eta \in \operatorname{LAut}(FI(P))$. Then there exists an order automorphism
$\rho$ of $(P,\le)$ such that $\eta (a) \in \widehat \rho(a) +Z(FI(P))$ for all $a\in FI(P)$.
\end{lemma}

\begin{proof}
We will make a more careful use of Equation \eqref{eq:inducedrho3}.

First, consider the element $a=e_x$ for some arbitrary element $x\in P$. With \eqref{eq:inducedrho3} we have
$e_{\rho(x)} = e_{\rho_a(x)}$. Thus $\rho(x) = \rho_a(x)$, and as $\rho_a$ preserves heights,
\begin{align} \label{eq:inducedrho5}
h(\rho(x))= h(\rho_a(x))=h(x) \mbox{\quad for all } x\in P.
\end{align}

Next, consider the element $a=e_x+e_y$ for distinct elements $x,y\in P$.  With \eqref{eq:inducedrho3} we have
$e_{\rho(x)}+e_{\rho(y)} = e_{\rho_a(x)}+e_{\rho_a(y)}$. Thus, either
\begin{align} \label{eq:inducedrho6}
\rho(x)=\rho_a(x),\, \rho(y)=\rho_a(y) \mbox{ \quad or \quad} \rho(x)=\rho_a(y),\, \rho(y)=\rho_a(x).
\end{align}
If $x<y$, then $h(x)<h(y)$, and the order automorphism $\rho_a$ of $(P,\le)$ yields
\[
h(\rho_a(x))=h(x) < h(y)=h(\rho_a(y)).
\]
However, $\rho(x)=\rho_a(y)$, $\rho(y)=\rho_a(x)$ yields with \eqref{eq:inducedrho3} that
\[
h(\rho_a(x))=h(\rho(y))=h(y) > h(x)=h(\rho(x))=h(\rho_a(y)),
\]
a contradiction. Thus $\rho(x)=\rho_a(x)$, $\rho(y)=\rho_a(y)$ holds, and
\[
\rho(x)= \rho_a(x) < \rho_a(y) = \rho(y).
\]
Similarly, $x> y$ yields $\rho(x)> \rho(y)$,
while $x,y$ incomparable yields $\rho(x),\rho(y)$ incomparable, and $\rho$ is an order automorphism.
\end{proof}

We can combine Lemmas \ref{lem:inducedrho1} and \ref{lem:inducedrho2} into one statement.

\begin{theorem} \label{thm:inducedrho}
Let $|P| < \mu_R$ and $\eta \in \operatorname{LAut}(FI(P))$. Then there exists an order automorphism
$\rho$ of $(P,\le)$ such that $\eta (a) \in \widehat \rho(a) +Z(FI(P))$ for all $a\in FI(P)$.
\end{theorem}

In particular, replacing $\eta$ with $\widehat \rho^{-1} \circ \eta$, we may, without loss of
generality, assume that the local automorphism $\eta$ induces the identity map on $FI(P)/Z$.

\subsection{Step 2: Splitting off $\psi_f$}\mbox{}\medskip

For $X\subseteq P$ define $e_{X}=\sum_{x\in X}e_{x}$. We start with some general observation
on the structure of idempotents.

\begin{lemma} \label{lem:idempow1}
Let $X\subseteq P$ and let
\[
a=\underset{y\le z}{\sum }a_{yz}e_{yz}=e_{X}+\underset{y<z}{\sum }a_{yz}e_{yz}\in e_{X}+Z(FI(P))
\]
be an idempotent element. Then $a_{yz}\neq 0$ implies $y\leq x\leq z$ for some $x\in X$.
\end{lemma}

\begin{proof}
Let $u<v$ with $a_{uv}\neq 0$ and consider the equation
\begin{align} \label{eq:idempow1}
a=a^{n}=\left(
\underset{x\in X}{\sum }e_{x}+\underset{y<z}{\sum }a_{yz}e_{yz}\right)^{n},
\end{align}
whose right-hand side must produce a nonzero coefficient at $e_{uv}$.
Pick
\begin{align} \label{eq:idempow2}
n>\left\vert \{(y,z):u\leq y<z\leq v,a_{yz}\neq 0\}\right\vert.
\end{align}
There exist elements $w(0),w(1),\ldots,w(n)\in P$ such that the product
\[
e_{uv}=e_{w(0)w(1)}e_{w(1)w(2)}e_{w(2)w(3)}\ldots e_{w(n-1)w(n)}
\]
makes a nonzero contribution
\[
\prod_{i=0}^{n-1} a_{w(i)w(i+1)}
\]
at $ e_{uv}$ to the right-hand side of \eqref{eq:idempow1}.
It follows that $u=w(0)$, $v=w(n)$ and $w(i)\leq w(i+1)$ for all $0\leq i<n$. If
\[
x\notin \{u=w(0),w(1),\ldots,w(n)=v\} \mbox{ for all } x\in X,
\]
then $w(i)<w(i+1)$ with $a_{w(i)w(i+1)}\not=0$ for all $0\leq i<n$. Hence,
\[
(w(i),w(i+1)) \in \{(y,z):u\leq y<z\leq v,a_{yz}\neq 0\}
\]
for all $0\leq i<n$, and
\[
\left\vert \{(y,z):u\leq y<z\leq v,a_{yz}\neq 0\}\right\vert \ge n,
\]
contradicting \eqref{eq:idempow2}. It follows that
$x\in \{w(0),w(1),\ldots,w(n)\}$ for some $x\in X$, and thus $u=w(0)\leq x\leq w(n)=v$.
\end{proof}

In the case of primitive idempotents we can be even more specific.

\begin{corollary} \label{cor:idempow2}
Fix $x\in P$ and let
\[
a=\underset{y\le z}{\sum }a_{yz}e_{yz}=e_{x}+\underset{y<z}{\sum }a_{yz}e_{yz}\in e_{x}+Z(FI(P))
\]
be an idempotent element. Then
\begin{itemize}
\item[(a)]
$a_{uv} = \left\{
  \begin{array}{rl}
     a_{ux}a_{xv}, & \mbox{\quad  for all $u\le x\le v$}, \\
     0, & \mbox{\quad else}.
  \end{array} \right.$
\item[(b)] In particular, $a=ae_xa$.
\end{itemize}
\end{corollary}

\begin{proof}
Note that
\[
a=a^2 = \left( \underset{u\le v}{\sum }a_{uv}e_{uv}\right)^2=\underset{u\le z\le v}{\sum }a_{uz}a_{zv}e_{uv}.
\]
Further note that $a_{uz}a_{zv}\neq 0$ with Lemma \ref{lem:idempow1} implies $x\in \lbrack u,z]\cap \lbrack z,v]=\{z\}$,
and thus
\[
a=\underset{u\le x\le v}{\sum }a_{ux}a_{xv}e_{uv}.
\]
For part (b), simply observe that
\[
a=\underset{u\le x\le v}{\sum }a_{ux}a_{xv}e_{ux}e_{xv}=\left( \underset{u\le x}{\sum }a_{ux}e_{ux}\right) \left( \underset{x\le v}{\sum }a_{xv}e_{xv}\right)=ae_x \cdot e_xa =ae_xa. \qedhere
\]
\end{proof}

We include the following result for a slightly different take on the same topic.

\begin{corollary} \label{cor:idempow2b}
Fix $x\in P$. Then the following holds.
\begin{itemize}
\item[(a)] If $a=be_xc$ with $b,c \in FI(P)$, then $a^2=a_xa$.
\item[(b)] The element $a \in e_{x}+Z(FI(P))$ is idempotent if and only if $a=be_xb$ for some $b \in FI(P)$.
\end{itemize}
\end{corollary}

\begin{proof}
For (a), simply note $a_x$=$b_xc_x$ and
\[
a^2=(be_xc)(be_xc)=b(e_xcbe_x)c=b(c_xb_xe_x)c=c_xb_x(be_xc)=a_xa.
\]
Part (b) is immediate from Corollary \ref{cor:idempow2}(b) and Corollary \ref{cor:idempow2b}(a).
\end{proof}

We are all set to show the orthogonality of the primitive idempotents $\eta(e_{x})$.

\begin{lemma} \label{lem:idempow3}
Let $\eta \in \operatorname{LAut}(FI(P))$ induce the identity map on $FI(P)/Z$.
Then $\eta(e_{x})\eta(e_{y}) = 0$ holds for all $x \not= y$.
\end{lemma}

\begin{proof}
We know that the idempotent elements $\eta (e_{x})$ and $\eta (e_{y})$
are of the forms
\[
\eta (e_{x})=\underset{u\le v}{\sum }\alpha_{uv}e_{uv}=e_{x}+\underset{u<v}{\sum }\alpha_{uv}e_{uv} \mbox{ and }
\eta (e_{y})=\underset{u\le v}{\sum }\beta_{uv}e_{uv}=e_{y}+\underset{u<v}{\sum }\beta_{uv}e_{uv}.
\]
Note that
\[
\eta(e_{x})\eta (e_{y}) = \left( \underset{u\le v}{\sum }\alpha_{uv}e_{uv}\right) \left( \underset{u\le v}{\sum }\beta_{uv}e_{uv}\right)
=\underset{u\le z\le v}{\sum }\alpha_{uz}\beta_{zv}e_{uv}.
\]
We distinguish the following two cases.

{\bf Case 1:} $x\nless y$.\\
Assume $\eta (e_{x})\eta (e_{y})\neq 0$. Pick $u\leq v\in P$ such that $(\eta (e_{x})\eta (e_{y}))_{uv}\neq 0$.
Then there exists some $u\le z\le v$ such that $\alpha_{uz}\neq 0\neq \beta_{zv}$.
By Lemma \ref{lem:idempow1} we get $u\leq x\leq z\leq y\leq v$ and thus $x\leq y$, a contradiction.

{\bf Case 2:} $x< y$.\\
Since $\eta$ preserves idempotents, considering the idempotents $e_x$, $e_y$ and $e_x+e_y$, we have
\begin{eqnarray*}
\eta(e_{x})+\eta(e_{y}) &=& \eta(e_{x}+e_{y})=\eta(e_{x}+e_{y})^{2}=(\eta(e_{x})+\eta(e_{y}))^2\\
&=& \eta(e_{x})^2+\eta(e_{y})^2+\eta(e_{x})\eta(e_{y})+\eta(e_{y})\eta(e_{x})\\
&=& \eta(e_{x})+\eta(e_{y})+\eta(e_{x})\eta(e_{y})+\eta(e_{y})\eta(e_{x}),
\end{eqnarray*}
 and
 \[
 \eta(e_{x})\eta(e_{y})+\eta(e_{y})\eta(e_{x})=0
 \]
 follows. By Case 1, we have $\eta(e_{y})\eta(e_{x})=0$,
 and thus $\eta(e_{x})\eta(e_{y})=0$.
\end{proof}

We want to strengthen Lemma \ref{lem:idempow3} to include $\eta(e_{x})\eta(e_{Y}) = \eta(e_{Y})\eta(e_{x})=0$
for suitably chosen subsets $Y\subseteq P$. This will need a little bit of technical preparation.
As usual, a subset $S$ of a poset $(P,\le)$ has the \emph{ascending chain condition} (acc) if it contains
no infinite strictly ascending chain. Similarly, $S$ has the \emph{descending chain condition} (dcc) if it contains
no infinite strictly descending chain.

\begin{lemma} \label{lem:idempow4}
Let a poset $(P,\le)$ and an infinite set $S\subseteq P$ be given. Then there exists a sequence $(y_i)_{i \in \omega}$
of elements in $S$ that is either strictly ascending, or strictly descending, or consisting of pairwise incomparable
elements in $(P,\le)$.
\end{lemma}

\begin{proof}
Let $S$ be an infinite subset of $P$. Assume that $S$ contains no infinite strictly
ascending or strictly descending chains. Then $S$ has the acc and dcc, and so does any
subset of $S$.

Let $S_0=S$, and let $\max(S_0)$ denote the set of elements maximal in $S_0$.
With acc, we have $\max(S_0)\not= \emptyset$.
Assume $\max(S_0)$ is finite. Then there exists some $m_{0}\in\max(S_0)$ such that
$S_{1}=\{x\in S_0:x<m_{0}\}$ is infinite. Assume $\max(S_{1})\not= \emptyset$ is finite. Then
there exists some $m_{1}\in\max(S_{1})$ such that $S_{2}=\{x\in S_{1} :x<m_{1}\}$ is infinite.
Continue the process. If $\max(S_{i})$ is finite for all $i$, then we
obtain an infinite sequence $m_{0}>m_{1}>\ldots >m_{i}>m_{i+1}>\ldots $, contradicting dcc.
Thus we must have encountered some $i$ such that $\max(S_{i})$
is infinite, and thus an infinite set of pairwise incomparable elements has been found.
\end{proof}

We can put the constructed sequence $(y_i)_{i \in \omega}$ to some good use to generalize the argumentation
of Lemma \ref{lem:idempow3}.

\begin{lemma} \label{lem:idempow5}
Let $\eta \in \operatorname{LAut}(FI(P))$ induce the identity map on $FI(P)/Z$. Then for every infinite set
$S\subseteq P$ there exists some infinite set $Y\subseteq S$ such that
\[
\eta(e_{x})\eta(e_{Y\setminus \{x\}}) = \eta(e_{Y\setminus \{x\}})\eta(e_{x})=0
\]
for all $x \in Y$.
\end{lemma}

\begin{proof}
With Lemma \ref{lem:idempow4}, we can choose a sequence $(y_i)_{i \in \omega}$ of elements in $S$ that is
either strictly ascending, or strictly descending, or consisting of pairwise incomparable elements in $(P,\le)$.
Let $Y=\{y_i \mid i \in \omega \}\subseteq S$. We distinguish the following three cases.

{\bf Case 1:} the elements $y_i$ are pairwise incomparable.\\
Assume $\eta(e_{x})\eta(e_{Y\setminus \{x\}})\not=0$ for some $x \in Y$. Pick $u\leq v\in P$ with $(\eta(e_{x})\eta(e_{Y\setminus \{x\}}))_{uv}$ $\neq 0$.
Then there exists some $u\le z\le v$ such that $(\eta(e_{x}))_{uz}\neq 0\neq (\eta(e_{Y\setminus \{x\}}))_{zv}$.
By Lemma \ref{lem:idempow1} we get $u\leq x\leq z\leq y\leq v$ for some $x\not= y \in Y$. Thus $x< y$, contradicting $x,y$ being incomparable.
Similarly, $\eta(e_{Y\setminus \{x\}})\eta(e_{x})=0$ follows.

{\bf Case 2:} the sequence $(y_i)_{i \in \omega}$ is strictly ascending.\\
Assume $\eta(e_{Y\setminus \{x\}})\eta(e_{x})\not=0$ for some $x=y_j \in Y$. With $Y'=\{y_i \mid i > j \}$ and Lemma \ref{lem:idempow3} we have
\begin{eqnarray*}
0\not= \eta(e_{Y\setminus \{x\}})\eta(e_{x}) &=& \eta\left(e_{Y'}+\sum_{i=0}^{j-1} e_{y_i}\right)\eta(e_{y_j})\\
&=& \eta(e_{Y'})\eta(e_{y_j}) +\sum_{i=0}^{j-1} \eta(e_{y_i})\eta(e_{y_j})= \eta(e_{Y'})\eta(e_x).
\end{eqnarray*}
Pick $u\leq v\in P$ with $(\eta(e_{Y'})\eta(e_{x}))_{uv}\neq 0$.
Then there exists some $u\le z\le v$ such that $(\eta(e_{Y'}))_{uz} \neq 0\neq (\eta(e_{x}))_{zv}$.
By Lemma \ref{lem:idempow1} we get $u\leq y\leq z\leq x\leq v$ for some $x\not= y= y_k \in Y'$. Thus $y_k=y<x= y_j$ with $k>j$,
contradicting $(y_i)_{i \in \omega}$ strictly ascending. This shows $\eta(e_{Y\setminus \{x\}})\eta(e_{x})=0$.

Since $\eta$ preserves idempotents, considering the idempotents $e_x$, $e_{Y\setminus \{x\}}$ and $e_Y$, we have
\begin{eqnarray*}
\eta(e_x)+\eta(e_{Y\setminus \{x\}}) &=& \eta(e_x+e_{Y\setminus \{x\}}) = \eta(e_Y)=\eta(e_Y)^2=\eta(e_x+e_{Y\setminus \{x\}})^{2}\\
&=& \eta(e_{x})^2+\eta(e_{Y\setminus \{x\}})^2+\eta(e_{x})\eta(e_{Y\setminus \{x\}})+\eta(e_{Y\setminus \{x\}})\eta(e_{x})\\
&=& \eta(e_{x})+\eta(e_{Y\setminus \{x\}})+\eta(e_{x})\eta(e_{Y\setminus \{x\}})+\eta(e_{Y\setminus \{x\}})\eta(e_{x}),
\end{eqnarray*}
 and
 \[
 \eta(e_{x})\eta(e_{Y\setminus \{x\}})+\eta(e_{Y\setminus \{x\}})\eta(e_{x})=0
 \]
 follows. Thus, $\eta(e_{Y\setminus \{x\}})\eta(e_{x})=0$ implies $\eta(e_{x})\eta(e_{Y\setminus \{x\}})=0$.

{\bf Case 3:} the sequence $(y_i)_{i \in \omega}$ is strictly descending.\\
This case is handled similar to Case 2.
\end{proof}

As an immediate consequence, we note the following.

\begin{theorem} \label{thm:idempow6}
Let $\eta \in \operatorname{LAut}(FI(P))$ induce the identity map on $FI(P)/Z$. Then for every infinite set
$S\subseteq P$ there exists some infinite set $Y\subseteq S$ such that
\[
\eta(e_{x})\eta(e_{Y}) = \eta(e_{Y})\eta(e_{x})=\eta(e_{x})
\]
for all $x \in Y$.
\end{theorem}

\begin{proof}
Choosing $Y$ as in Lemma \ref{lem:idempow5}, we have
\begin{eqnarray*}
\eta(e_{x})\eta(e_{Y})&=&\eta(e_{x})\eta(e_x+e_{Y\setminus \{x\}})=\eta(e_{x})(\eta(e_x)+\eta(e_{Y\setminus \{x\}}))\\
&=& \eta(e_{x})^2 +\eta(e_x)\eta(e_{Y\setminus \{x\}})=\eta(e_{x})^2=\eta(e_{x}),
\end{eqnarray*}
and $\eta(e_{Y})\eta(e_{x})=\eta(e_{x})$ follows similarly.
\end{proof}

We will next introduce an element $\beta\in I(P)$, in terms of $\eta$, crucial
to splitting off an inner automorphism $\psi_{\beta}$ of $\eta$. Of course, it
is essential to prove that $\beta\in FI(P)$. This requires quite some work
even if $\eta$ is an automorphism, cf. \cite{Khrip}. We have to develop some new ideas
to obtain the same result for the \textbf{local} automorphism $\eta$.

\begin{lemma} \label{lem:beta2}
Let $\eta \in \operatorname{LAut}(FI(P))$ induce the identity map on $FI(P)/Z$. Then
\[
\beta =\sum_{x\in P} e_{x}\eta (e_{x})
\]
defines an element $\beta \in FI(P)$.
\end{lemma}

\begin{proof}
Evidently, $\beta_{xy} =(\eta (e_{x}))_{xy}$ for all $x<y \in P$ by definition, and $\beta$
is a well-defined element of the incidence space $I(P)$. We need to show $\beta \in FI(P)$.

Assume $\beta \notin FI(P)$. Then there exist $a<b\in P$ for which the set
\[
W=\{(u,v): a\leq u<v\leq b,\beta_{uv}\neq 0\}
\]
is infinite. Let
\[
U=\{u\in P:\exists\, v\in P \mbox{ with } (u,v)\in W\},
\]
and for all $u\in U$ let
\[
S_u=\{v\in P:(u,v)\in W\}\neq \emptyset.
\]

If $S_u$ is infinite for some $u \in U$, then $\beta_{uv}=(\eta (e_{u}))_{uv}\neq 0$
for all $a\le v\le b$ with $v\in S_u$, a contradiction to $\eta (e_{u})\in FI(P)$.
Thus, $S_{u}\neq \emptyset$ is a finite set for all $u \in U$, and
\[
T_u=\{v\in S_u:v \mbox{ is minimal in } S_u\} \neq \emptyset.
\]
For $v\in T_u$ holds $\beta_{uv}=(\eta (e_{u}))_{uv}\neq 0$ but $(\eta (e_{u}))_{ut}=0$
for all $u<t<v$.

If the set $U$ is finite, then there exists some $u\in U$ with $S_u$ infinite, contradiction.
Thus $U$ is infinite, and we will distinguish between the following two cases.

{\bf Case 1:} $I=\{u\in U:y\in T_u\}=\{u\in U:(\eta (e_{u}))_{uy}\neq 0\}$ is infinite for some $y\in {\bigcup}_{u\in U}T_u$.\\
We have $y\notin I$ by definition, and with Theorem \ref{thm:idempow6}
we may assume
\[
\eta (e_{x})=\eta (e_{x})\eta (e_{I})
\]
for all $x\in I$, replacing $I$ by a suitable infinite subset of $I$ if need be.
We compute and compare the $e_{xy}$-coordinates of the terms in the last equation:
\begin{eqnarray*}
0 &\neq& (\eta (e_{x}))_{xy}=\sum_{x\leq t\leq y}(\eta(e_{x}))_{xt}(\eta (e_{I}))_{ty}\\
&=& (\eta (e_{x}))_{x} (\eta (e_{I}))_{xy}+(\eta (e_{x}))_{xy}(\eta (e_{I}))_{y}+\sum_{x< t< y}(\eta(e_{x}))_{xt}(\eta (e_{I}))_{ty}
\end{eqnarray*}
As $\eta$ induces the identity map on $FI(P)/Z$, we have $(\eta (e_{x}))_{x}=1$ and
$(\eta (e_{I}))_{y}=(e_{I})_{y}=0$ since $y\notin I$. Furthermore, as $y \in T_x$, we have $(\eta(e_{x}))_{xt}=0$ whenever $x<t<y$.
We obtain
\[
0\neq (\eta (e_{x}))_{xy}=(\eta (e_I))_{xy}
\]
for all $a\le x<y \le b$ with $x\in I$, a contradiction to $\eta (e_I)\in FI(P)$.

{\bf Case 2:} $\{u\in U:v\in T_u\}$ is finite for all $v\in {\bigcup}_{u\in U}T_u$.\\
We construct recursively distinct elements $x_i, y_i$ $(i \in \omega)$ with $y_i \in T_{x_i}$.
Start with arbitrary elements $x_0 \in U$ and $y_0 \in T_{x_0}$.
Given elements  $x_j, y_j$ $(j \le i)$, choose
\[
x_{i+1} \in U \setminus \left( \{x_j,y_j : j\le i\} \cup \bigcup_{j\le i} \{u\in U:x_j\in T_u\} \cup \bigcup_{j\le i} \{u\in U:y_j\in T_u\} \right)
\]
and $y_{i+1} \in T_{x_{i+1}}$. We have $(x_i,y_i)\in W$ by definition. Let $J=\{x_{i}:i<\omega \}$.
With Theorem \ref{thm:idempow6} we may assume
\[
\eta (e_{x_i})=\eta (e_{x_i})\eta (e_{J})
\]
for all $i\in \omega$, replacing $J$ by a suitable infinite subset of $J$ if need be.
We compute and compare the $e_{x_iy_i}$-coordinates of the terms in the last equation:
\begin{eqnarray*}
0 &\neq& (\eta (e_{x_{i}}))_{x_{i} y_{i}}=\sum_{x_{i}\leq t\leq y_{i}} (\eta (e_{x_{i}}))_{x_{i}t}\,(\eta (e_{J}))_{ty_{i}}\\
&=& (\eta (e_{x_{i}}))_{x_{i}}(\eta (e_{J}))_{x_iy_{i}}+(\eta (e_{x_{i}}))_{x_{i}y_i}(\eta (e_{J}))_{y_i}
+\sum_{x_{i}< t< y_{i}} (\eta (e_{x_{i}}))_{x_{i}t}\,(\eta (e_{J}))_{ty_{i}}
\end{eqnarray*}
As $\eta$ induces the identity map on $FI(P)/Z$, we have $(\eta (e_{x_i}))_{x_i}=1$ and
$(\eta (e_{J}))_{y_i}=(e_{J})_{y_i}=0$ since $y_i\notin J$. Furthermore, as $y_i \in T_{x_i}$, we have $(\eta(e_{x_i}))_{x_it}=0$ whenever $x_i<t<y_i$.
We obtain
\[
0\neq (\eta (e_{x_{i}}))_{x_{i} y_{i}}=(\eta (e_{J}))_{x_iy_{i}}
\]
for infinitely many distinct pairs $(x_i,y_i)\in W$ with $a\le x_i<y_i \le b$, a contradiction to $\eta (e_J)\in FI(P)$.
\end{proof}

We are all set for the main result of this section.

\begin{theorem}
Let $\eta \in \operatorname{LAut}(FI(P))$ induce the identity map on $FI(P)/Z$.
Then there exists a unit $\beta \in FI(P)$ with $\eta (e_{x})=\psi_\beta (e_{x})$ for all $x\in P$.
\end{theorem}

\begin{proof}
With Lemma \ref{lem:beta2}, we can choose $\beta =\sum_{x\in P} e_{x}\eta (e_{x}) \in FI(P)$.
Note that
\[
\beta_{x} =(\eta (e_{x}))_{x}=1
\]
is a unit of $R$ for all $x\in P$, which makes $\beta$ a unit of $FI(P)$, cf. Section \ref{sec:fipsdef}.
Furthermore, with Lemma \ref{lem:idempow3} we have
\begin{eqnarray*}
\beta \eta(e_{x}) = \sum_{y\in P} e_{y}\eta (e_{y})\eta(e_{x}) &=& e_{x}\eta (e_{x})^2 = e_{x}\eta (e_{x})\\
&=& e_{x}^2 \eta (e_{x})= e_x \sum_{y\in P} e_y\eta (e_{y}) =e_{x} \beta
\end{eqnarray*}
for all $x \in P$. Hence, $\eta(e_{x}) = \beta^{-1}e_{x} \beta = \psi_\beta (e_{x})$.
\end{proof}

In particular, replacing $\eta$ with $\psi_\beta^{-1} \circ \eta$, we may, without loss of
generality, assume that the local automorphism $\eta$ induces the identity map on $FI(P)/Z$
with  $\eta (e_{x})=e_{x}$ for all $x\in P$.

\subsection{Step 3: Splitting off $M_\sigma$} \mbox{}\medskip

Splitting off a suitable Schur multiplication $M_\sigma$ will be a refreshingly simple task.

\begin{theorem} \label{thm:step3}
Let $\eta \in \operatorname{LAut}(FI(P))$ with $\eta(e_{x}) = e_{x}$ for all $x \in P$. Then the following holds.
\begin{itemize}
\item[(a)] For all $x\le y \in P$, $\eta( e_{xy}) = \sigma_{xy} e_{xy}$ with $\sigma_{xy} \in R$.
\item[(b)] For all $x\le y \in P$, $\sigma_{xy}$ is a unit of $R$. Moreover, $\sigma _{xx}=1$ for all $x\in P$.
\item[(c)] For all $x\leq y\leq z \in P$, $\sigma _{xz}=\sigma _{xy}\sigma _{yz}$.
\end{itemize}
In particular, $\sigma$ induces a Schur multiplication $M_\sigma$ with $\eta (e_{xy})=M_\sigma (e_{xy})$ for all $x\le y\in P$.
\end{theorem}

\begin{proof}
For (a), we have $\eta(e_{xx}) = e_{xx}\in Re_{xx}$. Thus, we may assume $x<y$.

Applying $\eta$ to the idempotent $e_{x} + e_{xy}$ gives
\begin{eqnarray*}
e_{x} + \eta(e_{xy}) &=& \eta(e_{x}) + \eta(e_{xy}) =\eta(e_{x} + e_{xy}) = \eta(e_{x} + e_{xy})^2 = (e_{x} + \eta(e_{xy}))^2\\
&=& e_{x}+\eta(e_{xy})^2+e_{x}\eta(e_{xy})+\eta(e_{xy})e_{x}.
\end{eqnarray*}
Moreover, choosing some $\varphi \in \operatorname{Aut}(FI(P))$ with $\eta(e_{xy})=\varphi(e_{xy})$, we have
$\eta(e_{xy})^2=\varphi(e_{xy})^2=\varphi(e_{xy}^2)=\varphi(0)=0$, and the last equation simplifies to
\begin{align} \label{eq:step3a}
\eta(e_{xy})=e_{x}\eta(e_{xy})+\eta(e_{xy})e_{x}.
\end{align}
Similarly, from the idempotent $e_{y} + e_{xy}$ we infer
\begin{align} \label{eq:step3b}
\eta(e_{xy})=e_{y}\eta(e_{xy})+\eta(e_{xy})e_{y}.
\end{align}
Comparing coordinates on both sides of Equation \eqref{eq:step3a}, we have $(\eta(e_{xy}))_{uv}=0$ for $u\le v \in P$ unless either
$u=x$ or $v=x$. Similarly, \eqref{eq:step3b} gives $(\eta(e_{xy}))_{uv}=0$ for $u\le v \in P$ unless either $u=y$ or $v=y$. Thus,
$(\eta(e_{xy}))_{xy}$ is the only possible nontrivial entry of $\eta(e_{xy})$, and $\eta(e_{xy})\in R e_{xy}$.

For (b), let $\eta( e_{xy}) = \sigma_{xy} e_{xy}$ and choose some $\varphi \in \operatorname{Aut}(FI(P))$ with
$\eta(e_{xy})=\varphi(e_{xy})$. With $a=\varphi^{-1}(e_{xy})\in FI(P)$ we have
\[
\varphi(e_{xy})=\eta(e_{xy})=\sigma_{xy} e_{xy}=\sigma_{xy} \varphi(a)= \varphi(\sigma_{xy}a).
\]
Hence $e_{xy}=\sigma_{xy}a$, and looking at the $e_{xy}$-coordinates of this equation gives $1=\sigma_{xy}a_{xy}$.
Thus, $\sigma_{xy}$ is a unit, with $\sigma _{xx}=1$ evident from $\eta(e_{xx}) = e_{xx}$.

For (c), with $\sigma_{yy}=1$ the statement trivially holds if $x=y$ or $y=z$, and we need to check
$\sigma _{xz}=\sigma _{xy}\sigma _{yz}$ only for the case $x<y<z$. For that purpose, we will
investigate the idempotent element $a=e_{y}+e_{xy}+e_{yz}+e_{xz}$. First, note that
\[
a=e_{y}+e_{xy}+e_{yz}+e_{xz}=(e_y+e_{xy})(e_y+e_{yz})=a e_y a
\]
confirms $a$ as an idempotent, cf. Corollary \ref{cor:idempow2b}. We have
\[
\eta(a) = \eta(e_{y})+\eta(e_{xy})+\eta(e_{yz})+\eta(e_{xz})
= e_y+\sigma_{xy}e_{xy}+\sigma_{yz}e_{yz}+\sigma_{xz}e_{xz},
\]
and after squaring
\begin{eqnarray*}
\eta(a) = \eta(a)^2 &=& (e_y+\sigma_{xy}e_{xy}+\sigma_{yz}e_{yz}+\sigma_{xz}e_{xz})^2\\
&=& e_y+\sigma_{xy}e_{xy}+\sigma_{yz}e_{yz}+\sigma_{xy}\sigma_{yz}e_{xz}.
\end{eqnarray*}
Comparing these last two equations, we infer $\sigma _{xz}=\sigma _{xy}\sigma _{yz}$.
\end{proof}

In particular, replacing $\eta$ with $M_\sigma^{-1} \circ \eta$, we may, without loss of
generality, assume that the local automorphism $\eta$ induces the identity map on $FI(P)/Z$
with  $\eta (e_{xy})=e_{xy}$ for all $x\le y\in P$. It remains to show $\eta =\operatorname{id}$.

\subsection{Step 4: Finish} \mbox{}\medskip

It will be convenient to talk about diagonal elements.

\begin{definition}
We call $d\in FI(P)$ \emph{diagonal} if $d_{xy}=0$ for all $x<y \in P$.
\end{definition}

The following lemma provides a first very powerful boost towards proving $\eta =\operatorname{id}$.

\begin{lemma} \label{lem:step4a}
Let $\eta \in \operatorname{LAut}(FI(P))$ induce the identity map on $FI(P)/Z$ with $\eta(e_{x}) = e_{x}$ for all $x \in P$.
Then $\eta(d)=d$ for all diagonal elements $d={\sum}_{x\in P} d_{x}e_{x}$.
\end{lemma}

\begin{proof}
We have $\eta(d)=d+j$ for some $j\in Z$. For a contradiction, let us assume that $j\neq0$. Thus, we can choose $u<v\in P$ with $j_{uv}\neq0$.
The set
\[
W=\{t: u< t \leq v,j_{ut}\neq 0\}
\]
is finite. Without loss of generality, we may assume $W=\{v\}$, replacing $v$ by $\min W$ if need be.
Thus, $j_{uv} \neq 0$ but
\begin{align} \label{eq:step4z}
j_{ut}=0 \mbox{ for all } u<t<v.
\end{align}
Let $E=P\setminus \{u,v\}$, and set $d^{(E)}={\sum}_{x\in E}d_{x}e_{x}$. We have
\begin{align} \label{eq:step4a}
d=d_{u}e_{u}+d_{v}e_{v}+d^{(E)}.
\end{align}
Applying $\eta$ gives
\begin{align} \label{eq:step4b}
\eta (d)=\eta (d_{u}e_{u}+d_{v}e_{v}+d^{(E)})=d_{u}e_{u}+d_{v}e_{v}+\eta (d^{(E)}).
\end{align}
On the other hand, we have
\begin{align} \label{eq:step4c}
\eta(d)=d+j=d_{u}e_{u}+d_{v}e_{v}+d^{(E)}+j,
\end{align}
and we infer
\begin{align} \label{eq:step4d}
\eta(d^{(E)})=d^{(E)}+j
\end{align}
from comparing \eqref{eq:step4b} and \eqref{eq:step4c}.

Choose some $\varphi \in \operatorname{Aut}(FI(P))$ with $\eta(d^{(E)})=\varphi(d^{(E)})$.
With Theorem~\ref{thm:charaut} we have $\varphi = \psi_{f} \circ M_{\sigma} \circ \widehat \rho$.
Note that $\widehat \rho (d^{(E)})$ is a diagonal element and that diagonal elements are fixed
under Schur multiplications. Hence
\begin{align} \label{eq:step4e}
\eta(d^{(E)})=(\psi_{f} \circ M_{\sigma})(\widehat \rho(d^{(E)}))=\psi_{f} (\widehat \rho(d^{(E)})) = f^{-1}(\widehat \rho(d^{(E)}))f,
\end{align}
and
\[
\eta(d^{(E)})+Z = f^{-1}(\widehat \rho(d^{(E)}))f +Z = \widehat \rho(d^{(E)}) +Z.
\]
As $\eta$ induce the identity map on $FI(P)/Z$, we infer $\widehat \rho(d^{(E)})=d^{(E)}$.
Thus, \eqref{eq:step4e} becomes
\[
\eta(d^{(E)})=f^{-1}d^{(E)}f
\]
for some unit $f\in FI(P)$. Together with \eqref{eq:step4d} we have
\begin{align} \label{eq:step4f}
d^{(E)}+j=f^{-1}d^{(E)}f.
\end{align}

Now consider the $e_{uv}$-coordinate of the last equation. We have
\[
0\neq j_{uv}=(d^{(E)}+j)_{uv}=(f^{-1}d^{(E)}f)_{uv}=\sum_{u\leq t\leq v,t\in E}(f^{-1})_{ut}d_{t}f_{tv}.
\]
Note that in the latter summation we actually have $u<t<v$ since $u,v\notin E$. This implies
\begin{align} \label{eq:step4g}
(f^{-1})_{uw}d_{w}f_{wv} \neq0 \mbox{ for some } u<w<v.
\end{align}
With  \eqref{eq:step4f} we get
\begin{align} \label{eq:step4h}
(d^{(E)}+j)(f^{-1}e_{w}f)=(f^{-1}d^{(E)}f)(f^{-1}e_{w}f)=f^{-1}d^{(E)}e_{w}f=f^{-1}d_{w}e_{w}f.
\end{align}
Considering the $e_{uv}$-coordinate of the last equation leads to
\[
(d^{(E)}f^{-1}e_{w}f+jf^{-1}e_{w}f)_{uv}=(d_{w}f^{-1}e_{w}f)_{uv}=(f^{-1})_{uw}d_{w}f_{wv} \ne 0.
\]
Note that $(d^{(E)}f^{-1}e_{w}f)_{uv}=0$ since $u\notin E$, and we get
\[
0\neq (jf^{-1}e_{w}f)_{uv}=\sum_{u<t\leq w}j_{ut}(f^{-1})_{tw}f_{wv},
\]
where we have a strict inequality $u<t$ since $j\in Z$. It follows that $j_{ut}\neq0$ for some
$u<t\leq w<v$, contradicting \eqref{eq:step4z}.
\end{proof}

For our final theorem we need one last crucial definition.

\begin{definition}
For any $u \le v \in P$ let
\[
L_{uv}=\{ a \in FI(P): a_{xy}=0 \mbox{ for all } u\le x\le y\le v\}.
\]
\end{definition}

We summarize some of the remarkable properties of $L_{uv}$ as a separate lemma.

\begin{lemma} \label{lem:step4b}
Let $\eta \in \operatorname{LAut}(FI(P))$ induce the identity map on $FI(P)/Z$.
Then $L_{uv} \vartriangleleft FI(P)$ is a two-sided ideal for all $u\le v \in P$,
and $\eta(L_{uv}) \subseteq L_{uv}$ holds.
\end{lemma}

\begin{proof}
First we show that $L_{uv}$ is a two-sided ideal.
Let $a \in L_{uv}$ and $\gamma\in FI(P)$. Then
\[
(\gamma a)_{xy}=\sum_{x\leq t \leq y}\gamma_{xt} a_{ty}
\]
for all $x\le y \in P$. If $u\leq x\leq y\leq v$, then $a_{ty}=0$ for all
$x\leq t \leq y$. Thus $(\gamma a)_{xy}=0$, and we infer $\gamma a \in L_{uv}$.
In a similar way, $a \gamma \in L_{uv}$ follows.

For $\eta(L_{uv}) \subseteq L_{uv}$, start with some $a \in L_{uv}$ such that
\begin{align} \label{eq:step4i}
a_x=0 \mbox{ if and only if } u \le x\le v.
\end{align}
Choose some $\varphi \in \operatorname{Aut}(FI(P))$ with $\eta(a)=\varphi(a)$.
With Theorem~\ref{thm:charaut} we have $\varphi = \psi_{f} \circ M_{\sigma} \circ \widehat \rho$.
By definition, $\psi_{f}(L_{uv})= f^{-1}L_{uv}f \subseteq L_{uv}$ and
$M_{\sigma}(L_{uv}) \subseteq L_{uv}$ are evident, and we only need to
show that $\widehat{\rho }(a)\in L_{uv}$. Note that $\eta$ but also both
$\psi_{f}$ and $M_{\sigma}$ induce the identity map on $FI(P)/Z$. Thus
\[
\sum_{x\in P} a_x e_{x} +Z=a+Z=\eta (a)+Z = \widehat{\rho} (a) +Z = \sum_{x\in P} a_x e_{\rho(x)} +Z= \sum_{x\in P} a_{\rho^{-1}(x)} e_x +Z,
\]
and $a_x=a_{\rho^{-1}(x)}$ follows for all $x\in P$. In particular, with \eqref{eq:step4i} we have
\[
u\le x \le v \iff a_x =0 \iff a_{\rho^{-1}(x)} =0 \iff u\le \rho^{-1}(x) \le v
\]
and
\[
u\le x \le y \le v \Longrightarrow u\le \rho^{-1}(x) \le \rho^{-1}(y) \le v \Longrightarrow a_{ \rho^{-1}(x) \rho^{-1}(y)}=0.
\]
Thus
\[
\widehat{\rho }(a) = \sum_{x\le y} a_{xy} e_{\rho(x)\rho(y)} = \sum_{x\le y} a_{\rho^{-1}(x) \rho^{-1}(y)} e_{xy} \in L_{uv}.
\]
This shows $\widehat{\rho }(a)\in L_{uv}$ and $\eta(L_{uv}) \subseteq L_{uv}$ under condition \eqref{eq:step4i}.

Now let $a$ be any element of $L_{uv}$. Note that
\[
d = \left(\sum_{x\in P}e_x - \sum_{u\le x\le v} e_x \right) -\sum_{x\in P} a_xe_x
\]
is a diagonal element in $L_{uv}$, and that
\[
d+a = \left(\sum_{x\in P}e_x - \sum_{u\le x\le v} e_x \right) +\sum_{x<y} a_{xy}e_{xy}\in L_{uv}
\]
satisfies condition \eqref{eq:step4i}. Thus $\eta(d+a) \in L_{uv}$, and with Lemma \ref{lem:step4a} also
\[
\eta(a)=\eta(d+a)-\eta(d)=\eta(d+a)-d \in L_{uv}.
\]
We infer that $\eta (L_{uv})\subseteq L_{uv}$ for all $u\leq v\in P$.
\end{proof}

We are all set to complete the proof of Theorem \ref{thm:mainlautoffip}.

\begin{theorem} \label{thm:step4finish}
Let $\eta \in \operatorname{LAut}(FI(P))$ induce the identity map on $FI(P)/Z$ with
$\eta( e_{xy}) = e_{xy}$ for all $x\le y \in P$. Then $\eta = \operatorname{id}_{FI(P)}$.
\end{theorem}

\begin{proof}
Let $a\in FI(P)$. For any $u\le v\in P$ holds
\[
a = \sum_{x\le y} a_{xy}e_{xy} \in \sum_{x \in P} a_{x}e_{x} + \sum_{u\le x\le y \le v} a_{xy}e_{xy} +L_{uv}.
\]
With Lemma \ref{lem:step4b}, we infer
\[
\eta(a) \in \eta\left( \sum_{x \in P} a_{x}e_{x} \right) + \eta\left(\sum_{u\le x\le y \le v} a_{xy}e_{xy}\right) +L_{uv}.
\]
Note that the first sum describes a diagonal element, while the second sum is finite. Thus, with Lemma \ref{lem:step4a},
\[
\eta(a) \in \sum_{x \in P} a_{x}e_{x} + \sum_{u\le x\le y \le v} a_{xy}\eta(e_{xy}) +L_{uv}
= \sum_{x \in P} a_{x}e_{x} + \sum_{u\le x\le y \le v} a_{xy}e_{xy} +L_{uv}.
\]
In particular, $(\eta(a))_{uv}=a_{uv}$. As this holds for all $u\le v\in P$, we infer $\eta(a)=a$.
\end{proof}

\end{document}